\documentclass[11pt]{article}
\usepackage{amsmath,amsthm,amsfonts,latexsym,amsopn,verbatim,amscd,amssymb}
\usepackage{adjustbox}
\usepackage{lipsum}
\usepackage{cleveref}
\usepackage{enumerate}
\usepackage[left=2cm, right=2cm, top=2cm]{geometry}

\newtheorem{theorem}{Theorem}[section]
\newtheorem{Assumption}{Assumption}[section]

\newtheorem{Example}{Example}[section]

\newtheorem{lemma}{Lemma}[section]

\usepackage{tabularx}
\usepackage{amsmath,amsfonts,amssymb}
\usepackage{graphics, graphicx}
\usepackage{enumitem}
\usepackage{subcaption}
\usepackage{multirow}
\usepackage{authblk}
\usepackage{enumitem,amssymb}
\usepackage{graphicx,float}
\newlist{todolist}{itemize}{2}
\setlist[todolist]{label=$\square$}
\begin{document}
	\title{A parameter uniform hybrid approach for  singularly perturbed two-parameter parabolic problem with discontinuous data}
	\date{}
\author{Nirmali Roy$^{1*}$, Anuradha Jha$^{2}$}
\date{%
    $^{1,2}$Indian Institute of Information Technology Guwahati, Bongora, India,781015\\
   *Corresponding author. E-mail: {nirmali@iiitg.ac.in}
}
	\maketitle
	\begin{sloppypar}	%%%%%%%%%%%%%%%%%%%%%%%%%%%%%%%%%%%%%%%%%%%%%%%%%%%%%%
	\section*{Abstract} In this article, we address singularly perturbed two-parameter parabolic problem of the reaction-convection-diffusion type in two dimensions. These problems exhibit discontinuities in the source term and convection coefficient at particular domain points, which result in the formation of interior layers. The presence of two perturbation parameters leads to the formation of boundary layers with varying widths. Our primary focus is to address these layers and develop a scheme that is uniformly convergent. So we propose a hybrid monotone difference scheme for the spatial direction, implemented on a specially designed piece-wise uniform Shishkin mesh, combined with the Crank-Nicolson method on a uniform mesh for the temporal direction. The resulting scheme is proven to be uniformly convergent, with an order of almost two in the spatial direction and exactly two in the temporal direction. Numerical experiments support the theoretically proven higher order of convergence and shows that our approach results in better accuracy and convergence compared to other existing methods in the literature.  \\
		\textbf{Keywords}:  Interior layers, Hybrid Method, Singularly Perturbed, Parabolic problem,  Shishkin mesh.
%%%%%%%%%%%%%%%%%%%%%%%%%%%%%%%%%%%%%%%%%%%%%%%%%%%%%%%5
\section{Introduction}
We have considered a two-parameter singularly perturbed parabolic boundary value problem with non-smooth data:
\begin{equation}
\begin{aligned}
\label{twoparaparabolic}
&\mathcal{L} u(x,t)\equiv(\epsilon u_{xx}+\mu au_x-bu-u_t)(x,t)=f(x,t),~(x,t)\in(\Omega^{-}\cup\Omega^{+}),\\
&u(0,t)
=p(t),~ (0,t)\in\Gamma_l =\{(0,t)|0\le t\le T\};\\
&u(1,t)=r(t),(1,t)\in\Gamma_r= \{(1,t)|0\le t\le T\};\\
&u(x,0)=q(x), (x,0)\in\Gamma_b=\{(x,0)|0\le x\le1\} ,
\end{aligned}
\end{equation}
where \(0<\epsilon\ll1\) and \(0<\mu\le1\) are two singular perturbation parameters. Let \(d \in\Gamma=(0,1)\), \(\Gamma^{-}=(0,d)\), \(\Gamma^{+}=(d,1)\), \(\Gamma_{c}=\Gamma_{l}\cup\Gamma_{b}\cup\Gamma_{r}\), and \(\Omega=(0,1)\times(0,T]\), \(\Omega^{-}=(0,d)\times(0,T]\), \(\Omega^{+}=(d,1)\times(0,T]\). The convection coefficient \(a(x,t)\) and source term \(f(x,t)\) experience discontinuities at \((d,t)\in \Omega\) for all \(t\). Moreover, \(a(x,t)\le-\alpha_1<0\) for \((x,t)\in\Omega^-\) and \(a(x,t)\ge\alpha_2>0\) for \((x,t)\in\Omega^+\), where \(\alpha_1, \alpha_2\) are positive constants. These functions are sufficiently smooth on \((\Omega^- \cup \Omega^+)\). We assume the jumps of \(a(x,t)\) and \(f(x,t)\) at \((d,t)\) satisfy \(|[a](d,t)|<C\) and \(|[f](d,t)|<C\), where the jump of \(\omega\) at \((d,t)\) is \([\omega](d,t)=\omega(d+,t)-\omega(d-,t)\). Additionally, the coefficient \(b(x, t)\) is assumed to be a sufficiently smooth function on \(\Omega\) such that \(b(x,t)\ge\beta>0\). The boundary data \(p(t), r(t)\), and initial data \(q(x)\) are sufficiently smooth on the domain and satisfy the compatibility condition. Under these assumptions, the problem \eqref{twoparaparabolic} has a unique continuous solution in the domain \(\bar{\Omega}\). The solution contains strong interior layers at the points of discontinuity \((d,t)\) for all \(t \in (0,T]\) due to the discontinuity in the convection coefficient and source term. Additionally, the solution exhibits boundary layers at \(\Gamma_{l}\) and \(\Gamma_{r}\) due to the presence of perturbation parameters \(\epsilon\) and \(\mu\).\\
Singularly perturbed two-parameter parabolic problem are not extensively studied, and there is a scarcity of well-established numerical methods tailored to solve them effectively, especially in comparison to the abundance of methods available for one-parameter singularly perturbed parabolic problems. A number of researchers have developed numerical techniques for singularly perturbed two-parameter parabolic problems with smooth data  \cite{BDD, GKD, KS, MD, M, ER1, ZAM}. Finding a uniform numerical solution for singularly perturbed parabolic equations with discontinuous data is more challenging. For the problem \eqref{twoparaparabolic}, M. Chandru et al. in \cite{MC1} shown that the accuracy of the upwind scheme on Shishkin mesh in space and the backward Euler scheme on uniform mesh in time is nearly first-order. D. Kumar et al. in \cite{DP} used the Crank-Nicolson method in time on a uniform mesh and the upwind method in space on an appropriately defined Shishkin mesh to acheive second order accurate method in time and almost first order accurate in space. S. Singh et al. presented in \cite{SCK} an implicit scheme based on the Crank-Nicolson method in time and trigonometric B-spline basis functions on Shishkin mesh in space. This parameter uniform method is proved to be almost first order convergent in space and second order convergent in time. 

With uniform meshes, standard numerical approaches cannot yield effective solution approximations in layer regions. %In order to solve the problem of the form \eqref{twoparaparabolic}, in this article we used a hybrid scheme on a piece-wise uniform Shishkin mesh in space and Crank-Nicolson method in time to get a parameter uniform convergent method of almost order two in space and second-order in time in the discrete supremum norm. 
Stynes and Roos \cite{RS} introduced a hybrid technique for steady-state singular perturbation problems (SPPs) with continuous data which was almost of order two on a Shishkin mesh. The hybrid numerical scheme for solving singularly perturbed boundary value problems (BVPs) with discontinuous convection coefficients introduced by Cen in \cite{Cen} achieves nearly second-order accuracy across the entire domain \([0,1]\) when the perturbation parameter \(\epsilon\) satisfies \(\epsilon \le N^{-1}\). Also in \cite{MS}, K. Mukherjeee et. al. proposed a hybrid difference scheme on Shishkin mesh in space and Euler-backward difference in time for singularly perturbed one parabolic problem with non smooth data. They achieves the convergence almost two in space and first order in time. \\
In this paper, we have used the Crank-Nicolson scheme \cite{CN1} on time on a uniform mesh, and in space and a hybrid scheme on piece-wise uniform Shishkin mesh. The Crank-Nicolson method in time gives second order convergence in time in supremum norm. The hybrid scheme we propose combines a modified central difference scheme, the midpoint upwind scheme and upwind scheme. This blending offers a key advantage: it surpasses the limitations of the classical central difference scheme. The classical central difference scheme on a uniform mesh has been found to generate nonphysical oscillations in the discrete solution when $\epsilon$ is small, unless the mesh diameter is prohibitively small. On the other hand, the midpoint upwind approach shows second-order uniform convergence beyond the boundary layer region when used with uniform mesh. The upwind method though first order doesn't have spurious oscillations. Furthermore, the Shishkin mesh, uses piece-wise uniform meshes. Thus we choose a fine mesh within the layer regions  and a coarse mesh in the outer regions. Leveraging this characteristic, we utilize the second-order classical central difference scheme in the interior layer regions which is second order accurate under certain condition. The midpoint upwind scheme  and upwind scheme are used suitably in the outer region to get a monotone method. This approach allows us to accomplish nearly second-order convergence in space for the proposed hybrid scheme. Thus the overall method has  almost second order parameter uniform  convergence. \\
The rest of the paper is arranged as follows.. In Section 2, we present a \textit{priori} bounds on the analytical solution and its derivatives. The Shishkin mesh and the detailed construction of the suggested hybrid finite difference scheme are provided in Section 3. In section 4, truncation error of the scheme and uniform error bound is given. In Section 5, we conduct few numerical experiments to validate the theoretical findings and showcase the efficiency and accuracy of the proposed scheme. The primary findings are summarized in Section 6.\\
{\bf{Notation}}:  The norm used is the maximum norm given by
	$$\|u\|_{\Omega}=\max_{(x,t)\in \Omega}|u(x,t)|.$$
	%and the discrete maximum norm given by
%	$$\|U\|_{\bar{\Omega}^{N,M}}=\sup_{0\le i\le N,0\le j\le M}|U(x_{i},t_{j})|.$$
%	where $\bar{\Omega}^{N,M}$ is the discretized version of the domain $\bar{\Omega}$.\\
In this study, $C$, $C_{1}$, or $C_{2}$ are used as a positive constants that are independent of the perturbation parameters $\epsilon, \mu$, and mesh size.
 \section{Analytical aspects of the solution}
 The analytical properties of the solution to the problem \eqref{twoparaparabolic} are discussed in this section. The following provides details regarding the solution to equation \eqref{twoparaparabolic}, including the existence of a singularly perturbed solution, the minimum principle, stability bounds, and apriori bounds.
 \begin{lemma}\cite{DP}
     The solution $u(x,t)\in C^{0}(\bar{\Omega})\cap C^{1}(\Omega)\cap C^{2}(\Omega^{-}\cup \Omega^{+})$ exists for the \eqref{twoparaparabolic}.
 \end{lemma}
\begin{lemma}\cite{DP}
    If $u(x,t)\in C^{0}(\bar{\Omega})\cap C^{1}(\Omega)\cap C^{2}(\Omega^{-}\cup \Omega^{+})$ such that $u(x,t)\ge 0, \forall (x,t)\in \Gamma_{c},\mathcal{L}u(x,t)\le 0, \forall (x,t)\in (\Omega^{-}\cup \Omega^{+})$ and $[u_{x}](d,t) \le 0,t>0$ then $u(x,t)\geq 0 \; \forall ~(x,t)\in \bar{\Omega}$.
\end{lemma}
\begin{lemma}\cite{DP}
    The bounds on $u(x,t)$ is given by 
	$$\|u\|_{\bar{\Omega}}\le \|u\|_{\Gamma_{c}}+\frac{\|f\|_{\Omega^{-}\cup\Omega^{+}}}{\theta}.$$
 where $\theta=\min\{\alpha_{1}/d,\alpha_{2}/(1-d)\}.$
\end{lemma}
The following bounds are parameter-explicit and satisfied by the derivatives of the solution to this problem \eqref{twoparaparabolic}
\begin{lemma}\cite{MC1}
    The solution $u(x,t)$ and its derivatives $\forall(x,t)\in \Omega^{-}\cup \Omega^{+}$ satisfy the following bounds for non-negative integers k, m such that $1\le k+m\le3$:\\
    If $\sqrt{\alpha}\mu\le\sqrt{\rho \epsilon}$, then
    $$\bigg\|\frac{\delta^{k+m}u}{\delta x^{k}\delta t^{m}}\bigg\|\le C \epsilon^{-k/2}\max \bigg\{ \|u\|_{\bar{\Omega}},\sum_{i+2j=0}^{2} \epsilon^{i/2}\bigg\|\frac{\delta^{i+j}f}{\delta x^{i}\delta t^{j}}\bigg\|,\sum_{i=0}^{4} \bigg[\epsilon^{i/2}\bigg\|\frac{d^{i}p}{dt^{i} }\bigg\|_{\Gamma_{l}}+\bigg\|\frac{d^{i}q}{dx^{i} }\bigg\|_{\Gamma_{b}}+\bigg\|\frac{d^{i}r}{dt^{i} }\bigg\|_{\Gamma_{r}}\bigg]\bigg\}.$$
    If $\sqrt{\alpha}\mu>\sqrt{\rho \epsilon}$, then\\
    
    $\bigg\|\frac{\delta^{k+m}u}{\delta x^{k}\delta t^{m}}\bigg\|\le C \bigg(\frac{\mu}{\epsilon}\bigg)^{k}\bigg(\frac{\mu^{2}}{\epsilon}\bigg)^{m}\max \bigg\{ \|u\|_{\bar{\Omega}},\sum_{i+2j=0}^{2} \bigg(\frac{\epsilon}{\mu}\bigg)^{i}\bigg(\frac{\epsilon}{\mu^{2}}\bigg)^{j+1}\bigg\|\frac{\delta^{i+j}f}{\delta x^{i}\delta t^{j}}\bigg\|,\sum_{i=0}^{4} \bigg[\bigg(\frac{\epsilon}{\mu}\bigg)^{i}\bigg\|\frac{d^{i}p}{dt^{i} }\bigg\|_{\Gamma_{l}}+\bigg(\frac{\epsilon}{\mu^{2}}\bigg)^{j+1}\bigg\|\frac{d^{i}q}{dx^{i} }\bigg\|_{\Gamma_{b}}+\bigg(\frac{\epsilon}{\mu^{2}}\bigg)^{j+1}\bigg\|\frac{d^{i}r}{dt^{i} }\bigg\|_{\Gamma_{r}}\bigg]\bigg\},$
    where, $C$ is a constant independent of $\epsilon$ and $\mu$.
\end{lemma}
To obtain sharper bounds on the solution, the solution $u(x,t)$ of \eqref{twoparaparabolic} is decomposed into regular and layers components as $\displaystyle u(x,t)=v(x,t)+w_{l}(x,t)+w_{r}(x,t)$.
The regular component $v(x,t) $ satisfies the following equation:
\begin{equation}
\begin{aligned}
\label{v}
&\mathcal{L}v(x,t)=f(x,t),~~(x,t)\in\Omega^{-}\cup\Omega^{+},\\
&v(x,0)=u(x,0), ~\forall~x\in \Gamma^{-}\cup \Gamma^{+},\end{aligned}
\end{equation}
 and $v(0,t), v(1,t), v(d-,t), v(d+,t)$ are chosen suitably for all $t$ in $(0,T]$. Also,  $v(d-,t)= \lim\limits_{x \to d^-} v(x,t)$ and $v(d+,t)= \lim\limits_{x \to d^+} v(x,t).$
The regular component $v(x,t)$ can further be decomposed as
$$v(x,t)=\left\{
\begin{array}{ll}
\displaystyle v^{-}(x,t), & \hbox{ $(x,t)\in \Omega^-,$ }\\
v^{+}(x,t), & \hbox{ $(x,t)\in \Omega^+,$ }
\end{array}
\right.$$
where $v^{-}(x, t)$ and $v^{+}(x, t)$ are the left and right regular components respectively.

The singular components $w_{l}(x,t)$ and $w_{r}(x,t)$ are the solutions of
\begin{equation}
\begin{aligned}
\label{lw}
&\mathcal{L}w_{l}(x,t)=0,~~(x,t)\in \Omega^{-}\cup\Omega^{+}, \\
&w_{l}(0,t)=u(0,t)-v(0,t)-w_{r}(0,t),~\forall~ t\in(0,T],\\
&w_{l}(1,t) \text{ is chosen suitably}~\forall~ t\in(0,T],\\
&w_{l}(x,0)=0,~\forall~x\in \Gamma^{-}\cup \Gamma^{+},
\end{aligned}
\end{equation}
and 
\begin{equation}
\begin{aligned}
\label{rw}
&\mathcal{L}w_{r}(x,t)=0,~~(x,t)\in \Omega^{-}\cup\Omega^{+},\\
& w_{r}(0,t) ~\text{is chosen suitably} ~\forall~ t\in(0,T],\\
&w_{r}(1,t)=u(1,t)-v(1,t)-w_{l}(1,t),~\forall~t\in(0,T], \\
&w_{r}(x,0)=0,~\forall~x\in \Gamma^{-}\cup \Gamma^{+}, 
\end{aligned}
\end{equation}
respectively. The regular $v$ and layer $w_l, w_r$ components may be discontinuous at $(d,t), \forall \; t \in (0,T]$ but their sum $u$ is continuous at $(d,t), \forall \; t \in (0,T]$. 
Further, the singular components $w_{l}(x,t)$ and $w_{r}(x,t)$ are decomposed as
$$w_{l}(x,t)=\left\{
\begin{array}{ll}
\displaystyle w_{l}^{-}(x,t), & \hbox{ $(x,t)\in \Omega^-,$ }\\
w_{l}^{+}(x,t), & \hbox{ $(x,t)\in \Omega^+,$ }
\end{array}
\right.
w_{r}(x,t)=\left\{
\begin{array}{ll}
\displaystyle w_{r}^{-}(x,t), & \hbox{ $(x,t)\in \Omega^-,$ }\\
w_{r}^{+}(x,t), & \hbox{ $(x,t)\in \Omega^+.$ }
\end{array}
\right.$$
Hence, the unique solution $u(x,t)$ to \eqref{twoparaparabolic} is written as
$$u(x,t)=\left\{
\begin{array}{ll}
\displaystyle (v^{-}+w_{l}^{-}+w_{r}^{-})(x,t),\hspace{2.5cm} (x,t)\in\Omega^{-},\\
(v^{-}+w_{l}^{-}+w_{r}^{-})(d-,t)\\=(v^{+}+w_{l}^{+}+w_{r}^{+})(d+,t), ~~\hspace{1.5cm} (x,t)=(d,t),\forall~ t\in(0,T],\\
(v^{+}+w_{l}^{+}+w_{r}^{+})(x,t),\hspace{2.5cm}(x,t)\in\Omega^{+}.
\end{array}
\right.$$
 \begin{lemma}
 \label{bounds1}
    The bounds on the components $v(x,t),w_{l}(x,t),w_{r}(x,t)$ and their derivatives are given as follows for $\sqrt{\alpha}\mu\le\sqrt{\rho \epsilon}$ and $0\le i+2j\le 4$:
     $$\bigg\|\frac{\delta^{i+j}v}{\delta x^{i}\delta t^{j}}\bigg\|_{\Omega^{-}\cup
		\Omega^{+}}\le C (1+\epsilon^{(3-i)/2}),$$
     $$\bigg\|\frac{\delta^{i+j}w_{l}}{\delta x^{i}\delta t^{j}}\bigg\|_{\Omega^{-}\cup
		\Omega^{+}} \le C \epsilon^{-i/2}\left\{
	\begin{array}{ll}
	\displaystyle e^{-\theta_{2}x}, & \hbox{ $(x, t)\in \Omega^{-},$ }\\
	e^{-\theta_{1}(x-d)}, & \hbox{ $(x, t)\in \Omega^{+}$,}
	\end{array}
	\right.$$
	$$\bigg\|\frac{\delta^{i+j}w_{r}}{\delta x^{i}\delta t^{j}}\bigg\|_{\Omega^{-}\cup
		\Omega^{+}}\le C \epsilon^{-i/2} \left\{
	\begin{array}{ll}
	\displaystyle e^{-\theta_{1}(d-x)}, & \hbox{ $(x, t)\in \Omega^{-}$, }\\
	e^{-\theta_{2}(1-x)}, & \hbox{ $(x, t)\in \Omega^{+}$,}
	\end{array}
	\right.$$
	where
\begin{equation}
\label{theta1}
	\theta_{1}=\frac{\sqrt{\rho \alpha}}{\sqrt{\epsilon}},~~\theta_{2}=\frac{\sqrt{\rho \alpha}}{2\sqrt{\epsilon}},
\end{equation}
and $C$ is a constant independent of $\epsilon$ and $\mu$.
 \end{lemma}
 \begin{proof}
     The proof follows by following the approach given in Gracia and O'Riordan \cite{GO1}.
 \end{proof}
  \begin{lemma}
  \label{bounds2}
     The bounds on the components $v(x,t),w_{l}(x,t),w_{r}(x,t)$ and their derivatives are given as follows for $\sqrt{\alpha}\mu>\sqrt{\rho \epsilon}$ and $0\le i+2j\le 4$,:
     $$\bigg\|\frac{\delta^{i+j}v}{\delta x^{i}\delta t^{j}}\bigg\|_{\Omega^{-}\cup
		\Omega^{+}}\le C (1+(\epsilon/\mu)^{(3-i)}),$$
     $$\bigg\|\frac{\delta^{i+j}w_{l}}{\delta x^{i}\delta t^{j}}\bigg\|_{\Omega^{-}\cup
		\Omega^{+}} \le C \bigg(\frac{\mu}{\epsilon}\bigg)^{i}\left\{
	\begin{array}{ll}
	\displaystyle e^{-\theta_{2}x}, & \hbox{ $(x, t)\in \Omega^{-},$ }\\
	e^{-\theta_{1}(x-d)}, & \hbox{ $(x, t)\in \Omega^{+}$,}
	\end{array}
	\right.$$
	$$\bigg\|\frac{\delta^{i+j}w_{r}}{\delta x^{i}\delta t^{j}}\bigg\|_{\Omega^{-}\cup
		\Omega^{+}}\le C \bigg(\frac{1}{\mu}\bigg)^{i} \left\{
	\begin{array}{ll}
	\displaystyle e^{-\theta_{1}(d-x)}, & \hbox{ $(x, t)\in \Omega^{-}$, }\\
	e^{-\theta_{2}(1-x)}, & \hbox{ $(x, t)\in \Omega^{+}$,}
	\end{array}
	\right.$$
	where
\begin{equation}
\label{theta2}
	\theta_{1}=\frac{\alpha \mu}{\epsilon},~~\theta_{2}=\frac{\rho}{2\mu},
\end{equation}
and $C$ is a constant independent of $\epsilon$ and $\mu$.
 \end{lemma}
 \begin{proof}
     For the proof, refer to the idea given in Gracia and O'Riordan \cite{GO1}.
 \end{proof}
 \section{Discretization}
 In this section, we begin by demonstrating the semi-discretization of the problem in the temporal direction on a uniform mesh using the Crank-Nicolson method \cite{CN1}.
 For a fixed time $T$, the interval $[0,T]$ is partitioned uniformly as $\Lambda^{M}=\{t_{j}=j\Delta t:j=0,1,\ldots,M, \Delta t=\frac{T}{M}\}$. The semi-discretization yields the following system of linear ordinary differential equations:
\begin{align*}
&\epsilon U^{j+\frac{1}{2}}_{xx}(x)+\mu a^{j+\frac{1}{2}}(x)U^{j+\frac{1}{2}}_{x}(x)-b^{j+\frac{1}{2}}(x)U^{j+\frac{1}{2}}(x)=f^{j+\frac{1}{2}}(x)+\frac{U^{j+1}(x)-U^{j}(x)}{\Delta t},\\
&\hspace{3.3in} x\in  (\Omega^{-}\cup\Omega^{+}),~ 0\le j\le M-1,\\
&U^{j+1}(0)=u(0,t_{j+1}),~~~U^{j+1}(1)=u(1,t_{j+1}),~~0\le j\le M-1,\\
&U^{0}(x)=u(x,0),~~x\in\Gamma_{b},
\end{align*}
where $U^{j+1}(x)$ is the approximation of $u(x,t_{j+1})$ of eq. \eqref{twoparaparabolic} at $(j+1)$-th time level and $\displaystyle p^{j+\frac{1}{2}}=\frac{p^{j+1}(x)+p^{j}(x)}{2} $.
Upon simplification, we have
\begin{equation}
\label{semi-dis}
\left.
\begin{array}{ll}
\displaystyle \mathcal{\tilde{L}}U^{j+1}(x)=g(x,t_{j}), & \hbox{ $x\in  (\Gamma^{-}\cup\Gamma^{+}),~ 0\le j\le M-1$, }\\
U^{j+1}(0)=u(0,t_{j+1}), & \hbox{ $0\le j\le M-1$,}\\U^{j+1}(1)=u(1,t_{j+1}), & \hbox{ $0\le j\le M-1$,}\\
U^{0}(x)=u(x,0), & \hbox{ $x\in\Gamma,$ }
\end{array}
\right\}
\end{equation}
where the operator $\mathcal{\tilde{L}}$ is defined as\\
$$\mathcal{\tilde{L}}\equiv \epsilon \frac{d^{2}}{d x^{2}}+\mu a^{j+\frac{1}{2}}\frac{d}{d x}-c^{j+\frac{1}{2}}I,$$\\
and
\begin{align*}
    g(x,t_{j})&=2f^{j+\frac{1}{2}}(x)-\epsilon U^{j}_{xx}(x)-\mu a^{j+\frac{1}{2}}(x)U_{x}^{j}(x)+d^{j+\frac{1}{2}}(x)U^{j}(x),\\
c^{j+\frac{1}{2}}(x)&=b^{j+\frac{1}{2}}(x)+\frac{2}{\Delta t},\\
d^{j+\frac{1}{2}}(x)&=b^{j+\frac{1}{2}}(x)-\frac{2}{\Delta t}.
\end{align*}
The error in temporal semi-discretization is defined by 
$e^{j+1}= u(x,t_{j+1})- \hat{U}^{j+1}(x)$ where $u(x,t_{j+1})$ is the solution of  \eqref{twoparaparabolic}. $ \hat{U}^{j+1}(x)$  is the solution of semi-discrete scheme \eqref{semi-dis}, when $u(x,t_{j})$ is taken instead of $U^{j}$ to find solution at $(x,t_{j+1})$.
 \begin{theorem}
     The local truncation error $T_{j+1}=\mathcal{\tilde{L}}(e^{j+1})$ satisfies 
	$$\|T_{j+1}\|\le C(\Delta t)^{3},~~~0\le j\le M-1.$$
 \end{theorem}
 \begin{proof}
The result can be proved following the approach given in \cite{KTK}.
\end{proof}
\begin{theorem}
	The global  error $E^{j+1}=u(x,t_{j+1})-U^{j+1}(x)$ is estimated as 
	$$\|E^{j+1}\|\le C(\Delta t)^{2},~~~0\le j\le M-1,$$
	where $U^{j+1}(x)$ is the solution of \eqref{semi-dis}.
\end{theorem}
\begin{proof}
	The proof can be derived from the approach followed in \cite{KTK}.
\end{proof}
\subsection{Discretization in space}
We understand from the relevant literature on singular perturbation problems that a uniform mesh is insufficient for achieving uniform convergence of the Eq. \eqref{twoparaparabolic} due to the existence of layer regions. The differential equation \eqref{twoparaparabolic} exhibits strong interior layers including boundary layers near the boundaries. To resolve these layers, we have discretized this problem by using a piece-wise uniform Shishkin mesh.\\
Let the interior points of the spatial mesh be denoted by $\Gamma^{N}=\{x_{i}:1\le i \le \frac{N}{2}-1\}\cup\{x_{i}:\frac{N}{2}+1\le i \le N-1\}.$
The $\bar{\Gamma}^{N}=\{x_{i}\}_{0}^{N}\cup\{d\}$ denote the mesh points with $x_{0}=0,x_{N}=1$ and the point of discontinuity at point $x_{\frac{N}{2}}=d.$ We also introduce the notation $\Gamma^{N-}=\{x_{i}\}_{0}^{\frac{N}{2}-1}, \Gamma^{N+}=\{x_{i}\}^{N-1}_{\frac{N}{2}+1}$, $\Omega^{N-}=\Gamma^{N-}\times\Lambda^{M}$, $\Omega^{N+}=\Gamma^{N+}\times\Lambda^{M}$, $\Omega^{N,M}=\Gamma^{N}\times\Lambda^{M}, \bar{\Omega}^{N,M}=\bar{\Gamma}^{N}\times\Lambda^{M}$, $\Gamma_{c}^{N}=\bar{\Omega}^{N,M}\cap\Gamma_{c}$, $\Gamma_{l}^{N}=\Gamma_{c}^{N}\cap\Gamma_{l}$, $\Gamma_{b}^{N}=\Gamma_{c}^{N}\cap\Gamma_{b}$ and $\Gamma_{r}^{N}=\Gamma_{c}^{N}\cap\Gamma_{r}$.
The domain $[0,1]$ is subdivided into six sub-intervals as
\[\bar{\Gamma}=[0,\tau_{1}]\cup[\tau_{1},d-\tau_{2}]\cup[d-\tau_{2},d]\cup[d,d+\tau_{3}]\cup[d+\tau_{3},1-\tau_{4}]\cup[1-\tau_{4},1]= \bigcup_{i=1}^{6}I_{i}.\]
The transition points in $\bar{\Gamma}$ are:
$$\tau_{1}=\min\bigg\{\frac{d}{4},\frac{2}{\theta_{2}}\ln N\bigg\},~~~\tau_{2}=\min\bigg\{\frac{d}{4},\frac{2}{\theta_{1}}\ln N\bigg\},$$ 
$$\tau_{3}=\min\bigg\{\frac{1-d}{4},\frac{2}{\theta_{1}}\ln N\bigg\},~~~\tau_{4}=\min\bigg\{\frac{1-d}{4},\frac{2}{\theta_{2}}\ln N\bigg\}.$$
On the sub-intervals $[0,\tau_{1}],[d-\tau_{2},d],[d,d+\tau_{3}]$ and  $[1-\tau_{4},1]$ a uniform mesh of $\big(\frac{N}{8}+1\big)$ mesh points and on $[\tau_{1},d-\tau_{2}]$ and $[d+\tau_{3},1-\tau_{4}]$ a uniform mesh of $\big(\frac{N}{4}+1\big)$ mesh points is taken.
The mesh points are given by
\begin{equation*}
x_{i}=\left\{
\begin{array}{ll}
\displaystyle \frac{8i\tau_{1}}{N}, & 0\le i \le \frac{N}{8}, \vspace{.2cm} \\
\displaystyle\tau_{1}+\frac{4i(d-\tau_{1}-\tau_{2})}{N}, & \frac{N}{8}\le i \le \frac{3N}{8}, \vspace{.2cm}\\
\displaystyle d-\tau_{2}+\frac{8i\tau_{2}}{N}, & \frac{3N}{8}\le i \le \frac{N}{2}, \vspace{.2cm}\\
\displaystyle d+\frac{8i\tau_{3}}{N}, & \frac{N}{2}\le i \le \frac{5N}{8},\vspace{.2cm} \\
\displaystyle d+\tau_{3}+\frac{4i(1-d-\tau_{3}-\tau_{4})}{N}, & \frac{5N}{8}\le i \le \frac{7N}{8}, \vspace{.2cm} \\
\displaystyle1-\tau_{4}+\frac{8i\tau_{4}}{N}, &  \frac{7N}{8}\le i \le N.  
\end{array}
\right.
\end{equation*}
Let $h_{i}=x_{i}-x_{i-1}$ denote the step size and $h_{1}=\frac{8\tau_{1}}{N},h_{2}=\frac{4(d-\tau_{1}-\tau_{2})}{N},h_{3}=\frac{8\tau_{2}}{N},h_{4}=\frac{8\tau_{3}}{N},h_{5}=\frac{4(1-d-\tau_{3}-\tau_{4})}{N},h_{6}=\frac{8\tau_{4}}{N}$ denotes the step sizes in the interval $I_{1},I_{2},I_{3},I_{4},I_{5},I_{6}$, respectively.\\
\begin{Assumption}
    Throughout the paper we shall also assume that $\epsilon\le N^{-1}$ as is generally the case for discretization of convection-dominated problems.
\end{Assumption}
We have presented a hybrid difference approach for the discretization of the problem \eqref{twoparaparabolic} in the spatial variable. On the above piece-wise uniform Shishkin mesh, this hybrid difference scheme comprises of an upwind, a central difference, and a midpoint upwind scheme. Here, we introduce the hybrid difference method for the ratio $\sqrt{\alpha}\mu\le \sqrt{\rho \epsilon}$:
 \small
 \[\mathcal{L}_{h}^{N,M}U^{j+1}(x_{i})=\left\{
\begin{array}{ll}
\displaystyle \mathcal{L}_{c}^{N,M}U^{j+1}(x_{i}), & x_{i}\in  (0,1)\backslash \{\tau_{1},d-\tau_{2},d,d+\tau_{3},1-\tau_{4}\}, ~\text{with} ~\mu \|a\| h_{k}<2\epsilon,~1\le k\le 6,\vspace{0.2cm}\\
\mathcal{L}_{t}^{N,M}U^{j+1}(x_{\frac{N}{2}}), &  x_{i}=d ~\text{with}~h_{k}(\|b\|+\frac{2}{\Delta t})<2\mu \alpha,~k=3,4.
\end{array}
\right.\]
For the ratio $\sqrt{\alpha}\mu> \sqrt{\rho \epsilon}$, we use
\small
 \[\mathcal{L}_{h}^{N,M}U^{j+1}(x_{i})=\left\{
\begin{array}{ll}
\displaystyle \mathcal{L}_{m}^{N,M}U^{j+1}(x_{i}), & x_{i}\in  (0,\tau_{1})\cup(\tau_{1},d-\tau_{2}), ~\text{with} ~(\|b\|+\frac{2}{\Delta t})h_{k}\le2\alpha\mu,~k=1,2, \vspace{0.2cm}\\
\displaystyle \mathcal{L}_{u}^{N,M}U^{j+1}(x_{i}), & x_{i}\in  (0,\tau_{1})\cup(\tau_{1},d-\tau_{2}), ~\text{with} ~(\|b\|+\frac{2}{\Delta t})h_{k}>2\alpha\mu,~ k=1,2,  \vspace{0.2cm}\\
\displaystyle \mathcal{L}_{c}^{N,M}U^{j+1}(x_{i}), & x_{i}\in  (d-\tau_{2},d)\cup(d,d+\tau_{3}), ~\text{with} ~\mu \|a\| h_{k}<2\epsilon,  ~k=3,4,\vspace{0.2cm}\\
\displaystyle \mathcal{L}_{m}^{N,M}U^{j+1}(x_{i}), & x_{i}\in  (d+\tau_{3},1-\tau_{4})\cup(1-\tau_{4},1), ~\text{with} ~(\|b\|+\frac{2}{\Delta t})h_{k}\le2\alpha\mu,~k=5,6,  \vspace{0.2cm}\\
\displaystyle \mathcal{L}_{u}^{N,M}U^{j+1}(x_{i}), & x_{i}\in  (d+\tau_{3},1-\tau_{4})\cup(1-\tau_{4},1), ~\text{with} ~(\|b\|+\frac{2}{\Delta t})h_{k}>2\alpha\mu,~ k=5,6,  \vspace{0.2cm}\\
\mathcal{L}_{t}^{N,M}U^{j+1}(x_{\frac{N}{2}}), &  x_{i}=d ~\text{with}~h_{k}(\|b\|+\frac{2}{\Delta t})<2\mu \alpha,~k=3,4.
\end{array}
\right.\]
At the transition points for both case
\[\mathcal{L}_{h}^{N,M}U^{j+1}(x_{i})=\left\{
\begin{array}{ll}
\displaystyle \mathcal{L}_{m}^{N,M}U^{j+1}(x_{i}), &  (\|b\|+\frac{2}{\Delta t})h_{k}\le2\alpha\mu,~ k=2,3,5,6,\vspace{0.2cm}\\
\mathcal{L}_{u}^{N,M}U^{j+1}(x_{i}), &  ~~\text{otherwise.}
\end{array}
\right.\]
The fully discretize scheme is given by
\begin{equation}
\label{discrete problem}
    \mathcal{L}_{h}^{N,M}U^{j+1}(x_{i})\equiv r_{i}^{-}U^{j+1}(x_{i-1})+r_{i}^{c}U^{j+1}(x_{i})+r_{i}^{+}U^{j+1}(x_{i+1})=g^{j+1}(x_{i}).
\end{equation}
At the point of discontinuity $x_{N/2}=d$, we have used five point second order difference scheme
\begin{equation}\small
    \begin{aligned}
    \label{five point scheme}
        \mathcal{L}^{N,M}_{t}U^{j+1}(x_{\frac{N}{2}})=\frac{-U^{j+1}(x_{N/2+2})+4U^{j+1}(x_{\frac{N}{2}+1})-3U^{j+1}(x_{\frac{N}{2}})}{2h_{4}}\\-\frac{U^{j+1}(x_{\frac{N}{2}-2})-4U^{j+1}(x_{\frac{N}{2}-1})+3U^{j+1}(x_{\frac{N}{2}})}{2h_{3}}=0.
    \end{aligned}
\end{equation}
Associated with each of these finite difference operators, we have the following finite difference scheme
\begin{equation*}
\begin{aligned}
\label{schemes}
    \mathcal{L}_{c}^{N,M}U^{j+1}(x_{i})\equiv \epsilon \delta^{2}U^{j+1}(x_{i})+\mu a^{j+\frac{1}{2}}(x_{i})D^{0}U^{j+1}(x_{i})-b^{j+\frac{1}{2}}(x_{i})U^{j+1}(x_{i}),\\
 \mathcal{L}_{m}^{N,M}U^{j+1}(x_{i})\equiv \epsilon \delta^{2}U^{j+1}(x_{i})+\mu \Bar{a}^{j+\frac{1}{2}}(x_{i})D^{*}U^{j+1}(x_{i})-\Bar{b}^{j+\frac{1}{2}}(x_{i})\bar{U}^{j+1}(x_{i}),\\
 \mathcal{L}_{u}^{N,M}U^{j+1}(x_{i})\equiv \epsilon \delta^{2}U^{j+1}(x_{i})+\mu a^{j+\frac{1}{2}}(x_{i})D^{*}U^{j+1}(x_{i})-b^{j+\frac{1}{2}}(x_{i})U^{j+1}(x_{i}),
 \end{aligned}
\end{equation*}
and
\[g^{j+1}(x_{i})=\left\{
\begin{array}{ll}
\displaystyle 2f^{j+\frac{1}{2}}(x_{i})-\epsilon \delta_{x}^{2}U^{j}(x_{i})-\mu a^{j+\frac{1}{2}}(x_{i})D_{x}^{0}U^{j}(x_{i})+d^{j+\frac{1}{2}}(x)U^{j}(x_{i}), & \mathcal{L}_{h}^{N,M}=\mathcal{L}_{c}^{N,M}, \vspace{0.2cm}\\
\displaystyle 2f^{j+\frac{1}{2}}(x_{i})-\epsilon \delta_{x}^{2}U^{j}(x_{i})-\mu \bar{a}^{j+\frac{1}{2}}(x_{i})D_{x}^{*}U^{j}(x_{i})+\bar{d}^{j+\frac{1}{2}}(x)U^{j}(x_{i}), & \mathcal{L}_{h}^{N,M}=\mathcal{L}_{m}^{N,M}, \vspace{0.2cm}\\
2f^{j+\frac{1}{2}}(x_{i})-\epsilon \delta_{x}^{2}U^{j}(x_{i})-\mu a^{j+\frac{1}{2}}(x_{i})D_{x}^{*}U^{j}(x_{i})+d^{j+\frac{1}{2}}(x)U^{j}(x_{i}), &  \mathcal{L}_{h}^{N,M}=\mathcal{L}_{u}^{N,M},
\end{array}
\right.\]
 where,
 \[D^{+}U^{j+1}(x_{i})=\frac{U^{j+1}(x_{i+1})-U^{j+1}(x_{i})}{x_{i+1}-x_{i}},~~~~ D^{-}U^{j+1}(x_{i})=\frac{U^{j+1}(x_{i})-U^{j+1}(x_{i-1})}{x_{i}-x_{i-1}},\]
 \[D^{0}U^{j+1}(x_{i})=\frac{U^{j+1}(x_{i+1})-U^{j+1}(x_{i-1})}{h_{i+1}+h_{i}},~~~~\delta^{2}U^{j+1}(x_{i})=\frac{2(D^{+}U^{j+1}(x_{i})-D^{-}U^{j+1}(x_{i}))}{x_{i+1}-x_{i-1}},\]
\[D^{*}U^{j+1}(x_{i})=\left\{
\begin{array}{ll}
\displaystyle D^{-}U^{j+1}(x_{i}), & i<\frac{N}{2}, \vspace{0.2cm}\\
D^{+}U^{j+1}(x_{i}), &  i>\frac{N}{2},
\end{array}
\right.\] and
$\hbar_{i}=\frac{h_{i}+h_{i+1}}{2},~ \displaystyle \bar{w}^{j+1}(x_{i})=\frac{w^{j+1}(x_{i})+w^{j+1}(x_{i-1})}{2}$ in $\Omega^{N-}$,$~~\displaystyle \bar{w}^{j+1}(x_{i})=\frac{w^{j+1}(x_{i})+w^{j+1}(x_{i+1})}{2}$ in $\Omega^{N+}$.\\
The matrix associated with the above equations \eqref{discrete problem} is not an M-matrix. We transform the equation (\ref{five point scheme}) to establish the monotonicity property by estimating $U_{\frac{N}{2}-2}^{j+1}$ and $U_{\frac{N}{2}+2}^{j+1}$ for $\mathcal{L}_{t}^{N,M}U^{j+1}(x_{i})$.\\
For $\sqrt{\alpha}\mu\le\sqrt{\rho}\epsilon$ and $\sqrt{\alpha}\mu>\sqrt{\rho}\epsilon$, from the operator $\mathcal{L}_{c}^{N,M}$, we get\\ 
$U_{N/2-2}^{j+1}=\frac{2h_{3}^2}{2\epsilon-\mu h_{3} a_{N/2-1}^{j+\frac{1}{2}}}\bigg\{ g_{N/2-1}^{j+1}-U_{N/2-1}^{j+1}\bigg(\frac{-2\epsilon-h_{3}^2(b_{N/2-1}^{j+\frac{1}{2}}+\frac{2}{\Delta t})}{h_{3}^2}\bigg)-U_{N/2}^{j+1}\bigg(\frac{2\epsilon+\mu h_{3}a_{N/2-1}^{j+\frac{1}{2}}}{2h_{3}^2}\bigg)\bigg\}$\\
$U_{N/2+2}^{j+1}=\frac{2h_{4}^2}{2\epsilon+\mu h_{4} a_{N/2+1}^{j+\frac{1}{2}}}\bigg\{ g_{N/2+1}^{j+1}-U_{N/2+1}^{j+1}\bigg(\frac{-2\epsilon-h_{4}^2(b_{N/2+1}^{j+\frac{1}{2}}+\frac{2}{\Delta t})}{h_{4}^2}\bigg)-U_{N/2}^{j+1}\bigg(\frac{2\epsilon-\mu h_{4}a_{N/2+1}^{j+\frac{1}{2}}}{2h_{4}^2}\bigg)\bigg\}$\\
Inserting the above expression in $\mathcal{L}^{N,M}_{t}U^{j+1}(x_{N/2})$, we get a tridiagonal system of equations of the following form:\\
\begin{align*}
\mathcal{L}^{N,M}_{t}U^{j+1}_{N/2}&=\bigg(\frac{2\epsilon-h_{4}\mu a_{N/2+1}^{j+\frac{1}{2}}}{2\epsilon+h_{4}\mu a_{N/2+1}^{j+\frac{1}{2}}}h_{3}-3(h_{3}+h_{4})+\frac{2\epsilon+h_{3}\mu a_{N/2-1}^{j+\frac{1}{2}}}{2\epsilon-h_{3}\mu a_{N/2-1}^{j+\frac{1}{2}}}h_{4}\bigg)U^{j+1}_{N/2}+\bigg(\frac{-4\epsilon-2h_{4}^2 (b_{N/2+1}^{j+\frac{1}{2}}+\frac{2}{\Delta t})}{2\epsilon+h_{4}\mu a_{N/2+1}^{j+\frac{1}{2}}}+4\bigg)h_{3}U^{j+1}_{N/2+1}\\&+\bigg(\frac{-4\epsilon-2h_{3}^2 (b_{N/2-1}^{j+\frac{1}{2}}+\frac{2}{\Delta t})}{2\epsilon-h_{3}\mu a_{N/2-1}^{j+\frac{1}{2}}}+4\bigg)h_{4}U_{N/2-1}^{j+1}.\\&=\frac{2h_{3}^2h_{4}g^{j+1}_{N/2-1}}{2\epsilon-h_{3}\mu a_{N/2-1}^{j+\frac{1}{2}}}+\frac{2h_{4}^2h_{3}g_{N/2+1}^{j+1}}{2\epsilon+h_{4}\mu a_{N/2+1}^{j+\frac{1}{2}}}.
\end{align*}
The element of the system matrix $\mathcal{L}_{h}^{N,M}$ are as follows:\\
$$r_{i}^{-}=\frac{\epsilon}{h_{i}\hbar_{i}}-\frac{\mu a^{j+\frac{1}{2}}(x_{i})}{2\hbar_{i}}, ~r_{i}^{+}=\frac{\epsilon}{h_{i+1}\hbar_{i}}+\frac{\mu a^{j+\frac{1}{2}}(x_{i})}{2\hbar_{i}},~r_{i}^{c}=-r_{i}^{-}-r_{i}^{+}-\bigg(b^{j+\frac{1}{2}}(x_{i})+\frac{2}{\Delta t}\bigg),~~\text{if}~ \mathcal{L}_{h}^{N,M}=\mathcal{L}_{c}^{N,M}.$$

$$\text{In}~\Omega^{-},~r_{i}^{-}=\frac{\epsilon}{h_{i}\hbar_{i}}-\frac{\mu \Bar{a}^{j+\frac{1}{2}}(x_{i})}{h_{i}}-\bigg(\frac{\bar{b}^{j+\frac{1}{2}}(x_{i})}{2}+\frac{1}{\Delta t}\bigg), ~r_{i}^{+}=\frac{\epsilon}{h_{i+1}\hbar_{i}},~r_{i}^{c}=-r_{i}^{-}-r_{i}^{+}-\bigg(\bar{b}^{j+\frac{1}{2}}(x_{i})+\frac{2}{\Delta t}\bigg),~~\text{if}~ \mathcal{L}_{h}^{N,M}=\mathcal{L}_{m}^{N,M}.$$

$$\text{In}~\Omega^{+},~r_{i}^{-}=\frac{\epsilon}{h_{i}\hbar_{i}}, ~r_{i}^{+}=\frac{\epsilon}{h_{i+1}\hbar_{i}}-\frac{\mu \Bar{a}^{j+\frac{1}{2}}(x_{i})}{h_{i+1}}-\bigg(\frac{\bar{b}^{j+\frac{1}{2}}(x_{i})}{2}+\frac{1}{\Delta t}\bigg),~r_{i}^{c}=-r_{i}^{-}-r_{i}^{+}-\bigg(\bar{b}^{j+\frac{1}{2}}(x_{i})+\frac{2}{\Delta t}\bigg),~~\text{if}~ \mathcal{L}_{h}^{N,M}=\mathcal{L}_{m}^{N,M}.$$

$$\text{In}~\Omega^{-},~r_{i}^{-}=\frac{\epsilon}{h_{i}\hbar_{i}}-\frac{\mu a^{j+\frac{1}{2}}(x_{i})}{h_{i}}, ~r_{i}^{+}=\frac{\epsilon}{h_{i+1}\hbar_{i}},~r_{i}^{c}=-r_{i}^{-}-r_{i}^{+}-\bigg(b^{j+\frac{1}{2}}(x_{i})+\frac{2}{\Delta t}\bigg),~~\text{if}~ \mathcal{L}_{h}^{N,M}=\mathcal{L}_{u}^{N,M}.$$
$$\text{In}~\Omega^{+},~r_{i}^{-}=\frac{\epsilon}{h_{i}\hbar_{i}}, ~r_{i}^{+}=\frac{\epsilon}{h_{i+1}\hbar_{i}}+\frac{\mu a^{j+\frac{1}{2}}(x_{i})}{h_{i+1}},~r_{i}^{c}=-r_{i}^{-}-r_{i}^{+}-\bigg(b^{j+\frac{1}{2}}(x_{i})+\frac{2}{\Delta t}\bigg),~~\text{if}~ \mathcal{L}_{h}^{N,M}=\mathcal{L}_{u}^{N,M}.$$
$$\text{At $x_{N/2}=d$},~r_{N/2}^{-}=\bigg(\frac{-4\epsilon-2h_{3}^2 (b_{N/2-1}^{j+\frac{1}{2}}+\frac{2}{\Delta t})}{2\epsilon-h_{3}\mu a_{N/2-1}^{j+\frac{1}{2}}}+4\bigg)h_{4}, ~r_{N/2}^{+}=\bigg(\frac{-4\epsilon-2h_{4}^2 (b_{N/2+1}^{j+\frac{1}{2}}+\frac{2}{\Delta t})}{2\epsilon+h_{4}\mu a_{N/2+1}^{j+\frac{1}{2}}}+4\bigg)h_{3},$$ $$~r_{N/2}^{c}=\bigg(\frac{2\epsilon-h_{4}\mu a_{N/2+1}^{j+\frac{1}{2}}}{2\epsilon+h_{4}\mu a_{N/2+1}^{j+\frac{1}{2}}}h_{3}-3(h_{3}+h_{4})+\frac{2\epsilon+h_{3}\mu a_{N/2-1}^{j+\frac{1}{2}}}{2\epsilon-h_{3}\mu a_{N/2-1}^{j+\frac{1}{2}}}h_{4}\bigg).$$
To guarantee that the operator $\mathcal{L}_{h}^{N,M}$ is monotone, we impose the following assumption 
\begin{equation}
\label{condition}
    N(\ln N)^{-1}>16\max\bigg\{\frac{\|b\|}{\alpha},\frac{(\|b\|+\frac{2}{\Delta t})}{\alpha\rho}\bigg\}.
\end{equation}
To find the error estimates for the scheme \eqref{discrete problem} defined  above, we first decompose the discrete solution $U(x_{i},t_{j+1})$ into the discrete regular and discrete singular components. \\
Let
\[U^{j+1}(x_{i})=V^{j+1}(x_{i})+W_{l}^{j+1}(x_{i})+W_{r}^{j+1}(x_{i}).\]
The discrete regular components is
$$V^{j+1}(x_{i})=\left\{
\begin{array}{ll}
\displaystyle V^{-(j+1)}(x_{i}), & \hbox{ $(x_{i},t_{j+1})\in \Omega^{N-},$ }\\
V^{+(j+1)}(x_{i}),& \hbox{ $(x_{i},t_{j+1})\in \Omega^{N+},$ }
\end{array}
\right.$$
where $V^{-(j+1)}(x_{i})$ and $V^{+(j+1)}(x_{i})$ approximate $v^{-}(x_{i},t_{j+1})$ and $v^{+}(x_{i},t_{j+1})$ respectively. They satisfy the following equations:
\begin{equation}
\label{LV}
\begin{aligned}
	\left.
	\begin{array}{ll} 
\mathcal{L}^{N,M}_{h}V^{-(j+1)}(x_{i})=g(x_{i},t_{j+1}),~\forall (x_{i},t_{j+1})\in\Omega^{N-},\\V^{-(j+1)}(x_{0})=v^{-}(0,t_{j+1}), V^{-(j+1)}(x_{\frac{N}{2}})=v^{-}(d-,t_{j+1}),\\
\mathcal{L}^{N,M}_{h}V^{+(j+1)}(x_{i})=g(x_{i},t_{j+1}),~\forall (x_{i},t_{j+1})\in\Omega^{N+},\\
V^{+(j+1)}(x_{\frac{N}{2}})=v^{+}(d+,t_{j+1}), V^{+(j+1)}(x_{N})=v^{-}(1,t_{j+1}).
\end{array}
\right\}
\end{aligned}
\end{equation}
The discrete singular components $W_{l}^{j+1}(x_{i})$ and $W_{r}^{j+1}(x_{i})$ are also decomposed as:
$$W_{l}^{j+1}(x_{i})=\left\{
\begin{array}{ll}
\displaystyle W_{l}^{-(j+1)}(x_{i}), & \hbox{ $(x_{i},t_{j+1})\in \Omega^{N-},$ }\\
W_{l}^{+(j+1)}(x_{i}),& \hbox{ $(x_{i},t_{j+1})\in \Omega^{N+},$ }
\end{array}
\right.$$
$$
W_{r}^{j+1}(x_{i})=\left\{
\begin{array}{ll}
\displaystyle W_{r}^{-(j+1)}(x_{i}), & \hbox{$(x_{i},t_{j+1})\in \Omega^{N-},$ }\\
W_{r}^{+(j+1)}(x_{i}),& \hbox{$(x_{i},t_{j+1})\in \Omega^{N+},$ }
\end{array}
\right.$$
where $W_{l}^{-(j+1)}(x_{i}), W_{l}^{+(j+1)}(x_{i})$ approximates the layer components $w_{l}^{-}(x_{i},t_{j+1})$ and $w_{l}^{+}(x_{i},t_{j+1})$ and $W_{r}^{-(j+1)}(x_{i}), W_{r}^{+(j+1)}(x_{i})$ approximates the layer components $w_{r}^{-}(x_{i},t_{j+1})$ and $w_{r}^{+}(x_{i},t_{j+1})$ respectively.   
These components satisfy the following equations
\begin{equation}
\label{LW-}
\begin{aligned}
	\left.
	\begin{array}{ll} 
\mathcal{L}^{N,M}_{h}W_{l}^{-(j+1)}(x_{i})=0,~\forall (x_{i},t_{j+1})\in\Omega^{N-}, \\ W_{l}^{-(j+1)}(x_{0})=w_{l}^{-}(0,t_{j+1}), W_{l}^{-(j+1)}(x_{\frac{N}{2}})=w_{l}^{-}(d,t_{j+1}),\\
\mathcal{L}^{N,M}_{h}W_{l}^{+(j+1)}(x_{i})=0,~\forall ~(x_{i},t_{j+1})\in\Omega^{N+},\\
W_{l}^{+(j+1)}(x_{\frac{N}{2}})=w_{l}^{+}(d,t_{j+1}), W_{l}^{+(j+1)}(x_{N})=0.
\end{array}
\right\}
\end{aligned}
\end{equation}
\begin{equation}
\label{RW-}
\begin{aligned}
	\left.
	\begin{array}{ll} 
\mathcal{L}^{N,M}_{h}W_{r}^{-(j+1)}(x_{i})=0,~\forall (x_{i},t_{j+1})\in\Omega^{N-},\\
W_{r}^{-(j+1)}(x_{0})=0, W_{r}^{-(j+1)}(x_{\frac{N}{2}})=w_{r}^{-}(d,t_{j+1}),\\
\mathcal{L}^{N,M}_{h}W_{r}^{+(j+1)}(x_{i})=0,~\forall~ (x_{i},t_{j+1})\in\Omega^{N+},\\
W_{r}^{+(j+1)}(x_{\frac{N}{2}})=w_{r}^{+}(d+,t_{j+1}), W_{r}^{+(j+1)}(x_{N})=w_{r}(1,t_{j+1}).
\end{array}
\right\}
\end{aligned}
\end{equation}
Hence, the complete discrete solution $U^{j+1}(x_{i})$ defined by
$$U^{j+1}(x_{i})=\left\{
\begin{array}{ll}
\displaystyle (V^{-(j+1)}+W_{l}^{-(j+1)}+W_{r}^{-(j+1)})(x_{i}),\hspace{1.1cm} (x_{i},t_{j+1})\in\Omega^{N-},\\
(V^{-(j+1)}+W_{l}^{-(j+1)}+W_{r}^{-(j+1)})(d-)\\
=(V^{+(j+1)}+W_{l}^{+(j+1)}+W_{r}^{+(j+1)})(d+),\hspace{0.5cm}(x_{i},t_{j+1})=(d,t_{j+1}),\\
(V^{+(j+1)}+W_{l}^{+(j+1)}+W_{r}^{+(j+1)})(x_{i}),\hspace{1.1cm}(x_{i},t_{j+1})\in\Omega^{N+}.
\end{array}
\right.$$
\begin{lemma}
    Suppose a mesh function $Y(x_{i},t_{j})$ satisfies $Y(x_{i},t_{j})\ge 0,~\forall (x_{i},t_{j})\in \Gamma_{c}^{N}$ and $\mathcal{L}_{t}^{N,M}Y(x_{N/2})\le0 ~\forall~j=0,\ldots,M$. If $\mathcal{L}_{h}^{N,M}Y(x_{i},t_{j})\le0$ for all $(x_{i},t_{j})\in \Omega^{N-}\cup\Omega^{N+}$ then $Y(x_{i},t_{j})\ge0,~\forall  (x_{i},t_{j})\in\bar{\Omega}^{N,M}$.
\end{lemma}
\begin{proof}
   % The operator $\mathcal{L}^{N,M}_{h}Y(x_{i},t_{j+1})$ guarantee a M-matrix. So it is important to check the conditions $r_{i}^{-}>0,~r_{i}^{+}>0,~r_{i}^{-}+r_{i}^{c}+r_{i}^{+}<0$ for all the operator.\\
   % Let the first case $\sqrt{\alpha}\mu\le\sqrt{\rho\epsilon}:$ \\
    %The operator $\mathcal{L}^{N,M}_{c}$ is used in the region $(0,d)$ and $(d,1)$ except transition points with $\mu h_{k}\|a\|\le 2\epsilon$ where $k=1,2,3,4,5,6$ to guarantee $r^{-}_{i}>0$ and $r_{i}^{+}>0.$\\
   % In the case $\sqrt{\alpha}\mu\le\sqrt{\rho\epsilon}:$\\
   % The operator $\mathcal{L}^{N,M}_{m}$ is applied in the region $(0,\tau_{1})$ and $\mathcal{L}^{N,M}_{c}$ is used in $(d-\tau_{2},d)$. In the region $(\tau_{1},d-\tau_{2}),(d+\tau_{3},1-\tau_{4})$ and $(1-\tau_{4},1)$, the operator $\mathcal{L}^{N,M}_{m}$ is used with $\bigg(\|b\|+\frac{2}{\Delta t}\bigg)h_{k}<2\mu\alpha,~k=2,5,6$ otherwise $\mathcal{L}^{N,M}_{u}$ is used.\\
   The condition \eqref{condition} ensures that the matrix generated by $\mathcal{L}_{h}^{N,M}$  satisfies  $r_{i}^{-}>0,~r^{+}_{i}>0$ and $r_{i}^{-}+r_{i}^{c}+r_{i}^{+}<0$ in layer regions.  For the first case $\sqrt{\alpha}\mu\le\sqrt{\rho\epsilon},$
    the operator $\mathcal{L}^{N,M}_{c}$ is used in the region $(\tau_{1}, d-\tau_{2})$ and $(d+\tau_{3},1-\tau_{4})$ except transition points with $\mu h_{k}\|a\|\le 2\epsilon$ where $k=2,5$ to guarantee $r^{-}_{i}>0$ and $r_{i}^{+}>0.$\\
    In the case $\sqrt{\alpha}\mu\le\sqrt{\rho\epsilon},$
    the operator $\mathcal{L}^{N,M}_{m}$ is applied in the region $(\tau_{1},d-\tau_{2})$, and $(d+\tau_{3},1-\tau_{4})$ with $\big(\|b\|+\frac{2}{\Delta t}\big)h_{k}<2\mu\alpha,~k=2,5 $, otherwise $\mathcal{L}^{N,M}_{u}$ is used. At the transition point $\mathcal{L}^{N,M}_{m}$ is applied with $\big(\|b\|+\frac{2}{\Delta t}\big)h_{k}<2\mu\alpha,~k=2,3,5,6 $, otherwise $\mathcal{L}^{N,M}_{u}$ is used to guarantee $r^{-}_{i}>0$ and $r_{i}^{+}>0$. 
    At the point of discontinuity, the operator $\mathcal{L}_{t}^{N,M}$ is used with $\big(\|b\|+\frac{2}{\Delta t}\big)h_{k}<2\mu\alpha,~k=3,4$. The associated matrix with the operator $\mathcal{L}_{h}^{N,M}U^{j+1}(x_{i})$ is negative of a M- matrix. 
   The finite difference method is monotone and hence the finite difference operator satisfies discrete minimum principle.
\end{proof}
\begin{lemma}
    If $U(x_{i},t_{j+1}), (x_{i},t_{j+1})\in\bar{\Omega}^{N,M}$ is a mesh function  satisfying the difference scheme \eqref{discrete problem}, then $ \|U\|_{\bar{\Omega}^{N,M}}\le C$.
\end{lemma}
\begin{proof}
    Define the mesh function $\psi(x_{i},t_{j+1})=\phi(x_{i},t_{j+1})\pm U(x_{i},t_{j+1})$
    where $\phi(x_{i},t_{j+1})=\max \bigg\{|U^{j+1}(0)|,|U^{j+1}(1)|,\frac{\|g\|_{\bar{\Omega}^{N,M}}}{\beta}\bigg\}.$
    Clearly $\psi(x_{i},t_{j+1})\ge 0~\forall~(x_{i},t_{j+1})\in \Gamma_{c}^{N}$ and 
    $\mathcal{L}^{N,M}_{c}\psi(x_{i},t_{j+1})=-b(x_{i},t_{j+\frac{1}{2}})\phi(x_{i},t_{j+1})\pm \mathcal{L}^{N,M}_{c}U(x_{i},t_{j+1})\le0,$
    
    $\mathcal{L}^{N,M}_{m}\psi(x_{i},t_{j+1})=-\Bar{b}(x_{i},t_{j+\frac{1}{2}})\bar{\phi}(x_{i},t_{j+1})\pm \mathcal{L}^{N,M}_{m}U(x_{i},t_{j+1})\le0,$
    
    $\mathcal{L}^{N,M}_{u}\psi(x_{i},t_{j+1})=-b(x_{i},t_{j+\frac{1}{2}})\phi(x_{i},t_{j+1})\pm \mathcal{L}^{N,M}_{u}U(x_{i},t_{j+1})\le0.$\\
    At the point of discontinuity, 
    $\mathcal{L}^{N,M}_{t}\psi(x_{\frac{N}{2}},t_{j+1})=\mathcal{L}^{N,M}_{t}\phi(x_{\frac{N}{2}},t_{j+1})\pm \mathcal{L}^{N,M}_{t}U(x_{\frac{N}{2}},t_{j+1})\le0.$\\
    Therefore, $\psi(x_{i},t_{j+1})\ge0,~~\forall(x_{i},t_{j+1})\in\bar{\Omega}^{N,M}.$
    This leads to the required result
    $$\|U\|_{\bar{\Omega}^{N,M}}\le C$$
\end{proof}
\section{Truncation error analysis}
\begin{lemma}
\label{discrete singular}
    The singular component $W_{l}^{-}(x_{i},t_{j+1}), W_{l}^{+}(x_{i},t_{j+1})$, $W_{r}^{-}(x_{i},t_{j+1})$ and $ W_{r}^{+}(x_{i},t_{j+1})$ satisfy the following bounds
	$$|W_{l}^{-}(x_{i},t_{j+1})|\le C\gamma_{l,i}^{-},\quad\gamma_{l,i}^{-}=\prod_{n=1}^{i}(1+\theta_{2}h_{n})^{-1},~~\gamma_{l,0}^{-}=C_{1},~i=0,1,\ldots,\frac{N}{2},$$
	$$|W_{l}^{+}(x_{i},t_{j+1})|\le C\gamma_{l,i}^{+},\quad\gamma_{l,i}^{+}=\prod_{n=\frac{N}{2}+1}^{i}(1+\theta_{1}h_{n})^{-1},~~\gamma_{l,\frac{N}{2}}^{+}=C_{1},~i=\frac{N}{2}+1,\ldots,N,$$
	$$|W_{r}^{-}(x_{i},t_{j+1})|\le C\gamma_{r,i}^{-},\quad\gamma_{r,i}^{-}=\prod_{n=i+1}^{\frac{N}{2}}(1+\theta_{1}h_{n})^{-1},~~\gamma_{r,\frac{N}{2}}^{-}=C_{1},~i=0,1,\ldots,\frac{N}{2},$$
	$$|W_{r}^{+}(x_{i},t_{j+1})|\le C\gamma_{r,i}^{+},\quad\gamma_{r,i}^{+}=\prod_{n=i+1}^{N}(1+\theta_{2}h_{n})^{-1},~~\gamma_{r,N}^{+}=C_{1},~i=\frac{N}{2}+1,\ldots,N.$$
\end{lemma}
\begin{proof}
Define the barrier function for the left layer term as
\[\eta^{-(j+1)}_{l,i}= C\gamma_{l,i}^{-(j+1)} \pm W_{l}^{-(j+1)}(x_{i}),\quad  0\leq i \leq \frac{N}{2},~0\le j\le M-1 .\]	
where 
\[\gamma_{l,i}^{-(j+1)}=\left\{
\begin{array}{ll}
\displaystyle \prod_{k=1}^{i}(1+\theta_{2}h_{k})^{-1}, & 1\le i\le \frac{N}{2},\vspace{0.2cm}\\
1, &  i=0.
\end{array}
\right.\]
For the both cases $\sqrt{\alpha}\mu\le\sqrt{\rho\epsilon}$ and $\sqrt{\alpha}\mu>\sqrt{\rho\epsilon}$, for large enough $C$, $\eta^{-(j+1)}_{l,0}\geq 0,\eta^{-(0)}_{r,i}\geq 0$ and $\eta^{-(j+1)}_{l,N/2}\geq 0$. \\Now,
 \begin{eqnarray*}
 \mathcal{L}_{c}^{N,M} \eta^{-(j+1)}_{l,i}&=& \mathcal{L}_{c}^{N,M}  \gamma_{l,i}^{-(j+1)} \pm \mathcal{L}_{c}^{N,M} W_{l}^{-(j+1)}(x_{i})\\
 &=& \gamma_{l,i+1}^{-(j+1)}\left( 2 \epsilon \theta_{2}^2 \big(\frac{h_{i+1}}{h_{i+1}+{h_i}}-1\big)+2\epsilon \theta_{2}^2 -\mu a^{j+\frac{1}{2}}(x_i) \theta_{2}-\mu a^{j+\frac{1}{2}}(x_{i})\theta_{2}^{2}\frac{h_{i}h_{i+1}}{h_{i}+h_{i+1}}- b^{j+\frac{1}{2}}(x_i) \right)
 \end{eqnarray*}
 On simplification, we get
 $$\mathcal{L}_{c}^{N,M} \eta^{-(j+1)}_{l,i}\le \gamma_{l,i+1}^{-(j+1)}\left(2\epsilon \theta_{2}^2 -\mu a^{j+\frac{1}{2}}(x_i) \theta_2- b^{j+\frac{1}{2}}(x_i) \right)\le \gamma_{l,i+1}^{-(j+1)}(\rho|a^{j+\frac{1}{2}}(x_{i})|-b^{j+\frac{1}{2}}(x_{i}))\le0,$$
 $$
 \mathcal{L}_{m}^{N,M} \eta^{-(j+1)}_{l,i}\le \gamma_{l,i+1}^{-(j+1)}\left(2\epsilon \theta_{2}^2 -\mu \Bar{a}^{j+\frac{1}{2}}(x_i) \theta_2- \Bar{b}^{j+\frac{1}{2}}(x_i) \right)\le  \gamma_{l,i+1}^{-(j+1)}(\rho|\Bar{a}^{j+\frac{1}{2}}(x_{i})|-\Bar{b}^{j+\frac{1}{2}}(x_{i}))\le0$$
 $$
 \mathcal{L}_{u}^{N,M} \eta^{-(j+1)}_{l,i}\le \gamma_{l,i+1}^{-(j+1)}\left(2\epsilon \theta_{2}^2 -\mu a^{j+\frac{1}{2}}(x_i) \theta_2- b^{j+\frac{1}{2}}(x_i) \right)\le  \gamma_{l,i+1}^{-(j+1)}(\rho|a^{j+\frac{1}{2}}(x_{i})|-b^{j+\frac{1}{2}}(x_{i}))\le0.$$
 By discrete minimum principle for the continuous case \cite{RPS}, we obtain  \[\eta^{-(j+1)}_{l,i} \geq 0 \implies |W_{l}^{-(j+1)}(x_{i})|\leq C\prod_{k=1}^{i}(1+\theta_{2}h_{k})^{-1}, \; 1\leq i \leq \frac{N}{2},~0\le j\le M-1 .\]
Similarly, we prove the bound for $ W_{l}^{+(j+1)}(x_{i})$ for $\frac{N}{2}+1 \leq i \leq N-1,~0\le j\le M-1 .$\\
Now, to prove the bound for $W_{r}^{-(j+1)}(x_{i})$, we define the barrier function for the right layer component  as
\[\eta^{-(j+1)}_{r,i}= C\gamma_{r,i}^{-(j+1)} \pm W_{r}^{-(j+1)}(x_{i}),\quad  0\leq i \leq \frac{N}{2},~0\le j\le M-1.\]	
where 
\[\gamma_{r,i}^{-(j+1)}=\left\{
\begin{array}{ll}
\displaystyle \prod_{k=i+1}^{\frac{N}{2}}(1+\theta_{1}h_{k})^{-1}, & 0\le i< \frac{N}{2}, \vspace{0.2cm}\\
1, &  i=\frac{N}{2}.
\end{array}
\right.\]
For large enough $C$, $\eta^{-(j+1)}_{r,0}\geq 0,\eta^{-(0)}_{r,i}\geq 0$ and $\eta^{-(j+1)}_{r,N/2}\geq 0$.\\
Now,
\begin{eqnarray*}
	 \mathcal{L}_{c}^{N,M} \eta^{-(j+1)}_{r,i}&=& \mathcal{L}_{c}^{N,M}  \gamma_{r,i}^{-(j+1)} \pm \mathcal{L}_{c}^{N,M} W_{r}^{-(j+1)}(x_{i})\\
	&=& \frac{\gamma_{r,i}^{-(j+1)}}{1+\theta_1h_{i}}\left( 2\epsilon \theta_{1}^2 \big(\frac{h_{i}}{h_{i+1}+{h_i}}-1\big) +2\epsilon \theta_{1}^2 +\mu a^{j+\frac{1}{2}}(x_i) \theta_1 +\frac{\theta_{1}^{2}\mu a^{j+\frac{1}{2}}(x_{i})h_{i+1}}{h_{i}+h_{i+1}}- b^{j+\frac{1}{2}}(x_i) \right)\\
&\le & \frac{\gamma_{r,i}^{-(j+1)}}{1+\theta_1h_{i}}\left( 2 \epsilon \theta_{1}^2+\mu a^{j+\frac{1}{2}}(x_i) \theta_1 -b^{j+\frac{1}{2}}(x_i) \right), \quad \text{as } \frac{h_{i+1}}{h_{i+1}+{h_i}}-1\le0.\\
\end{eqnarray*}
For the case $\sqrt{\alpha}\mu\le \sqrt{\rho \epsilon}$, on simplification, we get
\begin{eqnarray*}
	 \mathcal{L}_{c}^{N,M} \eta^{-(j+1)}_{r,i}&\leq & \frac{\gamma_{r,i}^{-(j+1)}}{1+\theta_{1}h_{i}}\left(\rho|a^{j+\frac{1}{2}}(x_{i})|-b^{j+\frac{1}{2}}(x_{i}) \right) \leq 0,
\end{eqnarray*}
\begin{eqnarray*}
	 \mathcal{L}_{m}^{N,M} \eta^{-(j+1)}_{r,i}&\leq & \frac{\gamma_{r,i}^{-}}{1+\theta_{1}h_{i}}\left((\rho|\Bar{a}^{j+\frac{1}{2}}(x_{i})|-\Bar{b}^{j+\frac{1}{2}}(x_{i}))\right) \leq 0,
\end{eqnarray*}
\begin{eqnarray*}
	 \mathcal{L}_{u}^{N,M} \eta^{-(j+1)}_{r,i}&\leq & \frac{\gamma_{r,i}^{-}}{1+\theta_{1}h_{i}}\left((\rho|a^{j+\frac{1}{2}}(x_{i})|-b^{j+\frac{1}{2}}(x_{i})) \right) \leq 0,
\end{eqnarray*}
and for the case
 $\sqrt{\alpha}\mu > \sqrt{\rho \epsilon}$,
 \begin{eqnarray*}
	 \mathcal{L}_{c}^{N,M} \eta^{-(j+1)}_{r,i}&\leq & \frac{\gamma_{r,i}^{-(j+1)}}{1+\theta_1h_{i}}\left(\frac{\alpha \mu^2}{2\epsilon}(\alpha+a^{j+\frac{1}{2}}(x_{i}))-b^{j+\frac{1}{2}}(x_{i}) \right) \leq 0,
\end{eqnarray*}
\begin{eqnarray*}
	 \mathcal{L}_{m}^{N,M} \eta^{-(j+1)}_{r,i}&\leq & \frac{\gamma_{r,i}^{-}}{1+\theta_{1}h_{i}}\left(\frac{\alpha \mu^2}{2\epsilon}(\alpha+\Bar{a}^{j+\frac{1}{2}}(x_{i}))-\Bar{b}^{j+\frac{1}{2}}(x_{i}))\right) \leq 0,
\end{eqnarray*}
\begin{eqnarray*}
	 \mathcal{L}_{u}^{N,M} \eta^{-(j+1)}_{r,i}&\leq & \frac{\gamma_{r,i}^{-}}{1+\theta_{1}h_{i}}\left(
 \frac{\alpha \mu^2}{2\epsilon}(\alpha+a^{j+\frac{1}{2}}(x_{i}))-b^{j+\frac{1}{2}}(x_{i})) \right) \leq 0.
\end{eqnarray*}
By discrete minimum principle for the continuous case \cite{RPS}, we obtain  \[\eta^{-(j+1)}_{r,i} \geq 0 \implies |W_{r}^{-(j+1)}(x_{i})| \leq C\prod_{k=i+1}^{N/2}(1+\theta_{1}h_{k})^{-1}, \; 1\leq i \leq \frac{N}{2},~0\le j\le M-1.\]
Similarly, we prove the bound for $ W_{r}^{+(j+1)}(x_{i})$ for $\frac{N}{2}+1 \leq i \leq N-1,~0\le j\le M-1 .$
\end{proof}
The truncation error at the mesh point $(x_{i},t_{j+1})\in\Omega^{N,M}\backslash\{d\}$, when mesh is arbitrary we have:
\begin{align*}
    |\mathcal{L}_{h}^{N,M}(e(x_{i}))|\le\left\{\begin{array}{ll}
\displaystyle |(\mathcal{L}_{c}^{N,M}-\mathcal{L})y_{i}|\le\epsilon\hbar_{i}\|y_{xxx}\|+\mu\hbar_{i}\|a\|\|y_{xx}\|+\|y_{tt}\|\\
\displaystyle |(\mathcal{L}_{m}^{N,M}-\mathcal{L})y_{i}|\le\epsilon\hbar_{i}\|y_{xxx}\|+C\mu h_{i+1}^2(\|y_{xxx}\|+\|y_{xx}\|)+\|y_{tt}\|\\ 
\displaystyle  |(\mathcal{L}_{u}^{N,M}-\mathcal{L})y_{i}|\le\epsilon\hbar_{i}\|y_{xxx}\|+\mu h_{i+1}\|a\|\|y_{xx}\|+\|y_{tt}\|.
\end{array}
\right.
\end{align*}
On a uniform mesh with step size h, we get
\begin{align*}
    |\mathcal{L}_{h}^{N,M}(e(x_{i}))|\le\left\{\begin{array}{ll}
\displaystyle |(\mathcal{L}_{c}^{N,M}-\mathcal{L})y_{i}|\le\epsilon h^{2}\|y_{xxxx}\|+\mu h^{2}|a\|\|y_{xxx}\|+\|y_{tt}\|\\
\displaystyle |(\mathcal{L}_{m}^{N,M}-\mathcal{L})y_{i}|\le\epsilon h\|y_{xxx}\|+C\mu h_{i+1}^2(\|y_{xxx}\|+\|y_{xx}\|)+\|y_{tt}\|\\
\displaystyle  |(\mathcal{L}_{u}^{N,M}-\mathcal{L})y_{i}|\le\epsilon h^{2} \|y_{xxxx}\|+\mu h \|a\|\|y_{xx}\|+\|y_{tt}\|,
\end{array}
\right.
\end{align*}
where $C_{\|a\|,\|a'\|}$ is a positive constant depending on $\|a\|$ and $\|a'\|$.
\begin{lemma}
\label{regular part}
	The discrete regular component $V^{j+1}(x_{i})$  defined in \eqref{LV} and $ v(x,t)$ is solution of the problem \eqref{v}. So, the error in the regular component satisfies the following estimate for $\sqrt{\alpha}\mu \le \sqrt{\rho\epsilon}$ :
	$$\|V-v\|_{\Omega^{N-}\cup\Omega^{N+}}\le C(N^{-2}+\Delta t^{2}).$$
\end{lemma}
\begin{proof}
    The truncation error for the regular part of the solution $y(x,t)$ of the equation \eqref{twoparaparabolic} for the case $\sqrt{\alpha}\mu\le \sqrt{\rho \epsilon}$ when mesh is uniform $(\tau_{1}=\tau_{2}=\frac{d}{4})$ in the domain $(x_{i},t_{j+1})\in\Omega^{N-}$ is
\begin{align*}
\lvert \mathcal{L}_{h}^{N,M}(V^{-(j+1)}-v^{-(j+1)})(x_{i})\rvert &=\vert \mathcal{L}_{c}^{N,M}(V^{-(j+1)}-v^{-(j+1)})(x_{i})\rvert \\
&\le \bigg\lvert \epsilon \bigg(\delta^{2}-\frac{d^{2}}{dx^{2}}\bigg)v^{-(j+1)}(x_{i})\bigg\rvert +\mu \lvert a^{j+\frac{1}{2}}(x_{i})\rvert  \bigg\lvert \bigg(D^{0}-\frac{d}{dx}\bigg)v^{-(j+1)}(x_{i})\bigg\rvert\\&+\bigg\lvert \bigg(D_{t}^{-}-\frac{\delta}{\delta t}\bigg)v^{-(j+1)}(x_{i})\bigg\rvert\\
&\le\epsilon h^{2}\|v^{-}_{xxxx}\|+\mu h^{2}\|a\|\|v^{-}_{xxx}\|+C_{1}\Delta t^{2}\\&\le C(N^{-2}+\Delta t^{2}).\end{align*}
When mesh is non uniform in $(x_{i},t_{j+1})\in\Omega^{N-}$ for $\sqrt{\alpha}\mu \le \sqrt{\rho\epsilon}$:
\begin{align*}
\lvert \mathcal{L}_{c}^{N,M}(V^{-(j+1)}-v^{-(j+1)})(x_{i})\rvert&\le\epsilon \hbar_{i}\|v^{-}_{xxx}\|+\mu \hbar_{i}\|a\|\|v^{-}_{xx}\|+C_{1}\Delta t^{2} \le C(N^{-2}+\Delta t^{2}).
\end{align*}
For the case $\sqrt{\alpha}\mu> \sqrt{\rho \epsilon}$ when mesh is uniform $(\tau_{1}=\tau_{2}=\frac{d}{4})$ in the domain $(x_{i},t_{j+1})\in\Omega^{N-}$ is
\begin{align*}
\lvert \mathcal{L}_{h}^{N,M}(V^{-(j+1)}-v^{-(j+1)})(x_{i})\rvert &=\vert \mathcal{L}_{c}^{N,M}(V^{-(j+1)}-v^{-(j+1)})(x_{i})\rvert
\le\epsilon h^{2}\|v^{-}_{xxxx}\|+\mu h^{2}\|a\|\|v^{-}_{xx}\|+C_{1}\Delta t^{2}\\&\le C(N^{-2}+\Delta t^{2}).
\end{align*}
\begin{align*}
\lvert \mathcal{L}_{h}^{N,M}(V^{-(j+1)}-v^{-(j+1)})(x_{i})\rvert &=\vert \mathcal{L}_{m}^{N,M}(V^{-(j+1)}-v^{-(j+1)})(x_{i})\rvert \le C(N^{-2}+\Delta t^{2}).
\end{align*}
\begin{align*}
\lvert \mathcal{L}_{h}^{N,M}(V^{-(j+1)}-v^{-(j+1)})(x_{i})\rvert &=\vert \mathcal{L}_{u}^{N,M}(V^{-(j+1)}-v^{-(j+1)})(x_{i})\rvert \le C(N^{-2}+\Delta t^{2}).
\end{align*}
When mesh is non uniform,
\begin{align*}
\lvert \mathcal{L}_{c}^{N,M}(V^{-(j+1)}-v^{-(j+1)})(x_{i})\rvert&\le\epsilon \hbar_{i}\|v^{-}_{xxx}\|+\mu \hbar_{i}\|a\|\|v^{-}_{xx}\|+C_{1}\Delta t^{2} \le C(N^{-2}+\Delta t^{2}).
\end{align*}
\begin{align*}
\lvert \mathcal{L}_{m}^{N,M}(V^{-(j+1)}-v^{-(j+1)})(x_{i})\rvert\le \epsilon\hbar_{i}\|v^{-}_{xxx}\|+C\mu h_{i+1}^2(\|v^{-}_{xxx}\|+\|v^{-}_{xx}\|) \le C(N^{-2}+\Delta t^{2}).
\end{align*}
\begin{align*}
\lvert \mathcal{L}_{u}^{N,M}(V^{-(j+1)}-v^{-(j+1)})(x_{i})\rvert\le\epsilon\hbar_{i}\|v^{-}_{xxx}\|+C\mu h_{i+1}(\|v^{-}_{xx}\|)  \le C(N^{-2}+\Delta t^{2}).
\end{align*}
At the transition point $\tau_{1}$:
\begin{align*}
\lvert \mathcal{L}_{m}^{N,M}(V^{-(j+1)}-v^{-(j+1)})(x_{i})\rvert \le C(N^{-2}+\Delta t^{2}).
\end{align*}
\begin{align*}
\vert \mathcal{L}_{u}^{N,M}(V^{-(j+1)}-v^{-(j+1)})(x_{i})\rvert \le C(N^{-2}+\Delta t^{2}).
\end{align*}
Define the barrier function in $(x_{i},t_{j+1})\in\Omega^{N-}:$
$$\psi^{j+1}(x_{i})=C(N^{-2}+\Delta t^2)\pm(V^{-(j+1)}-v^{-(j+1)})(x_{i}).$$
For large C,  $\psi^{j+1}(0)\ge0, \; \psi^{j+1}(x_{\frac{N}{2}})\ge0,\psi^{0}(x_{i})\ge0$ and $\mathcal{L}_{h}^{N,M}\psi^{j+1}(x_{i})\le0$. 
Hence, using the approach given in \cite{FHMRS1}, we get $\psi^{j+1}(x_{i})\ge0$ and 
\begin{equation}
\label{V-}
\lvert (V^{-(j+1)}-v^{-(j+1)})(x_i)\rvert_{\Omega^{N-}} \le C(N^{-2}+\Delta t^{2}).
\end{equation}
Similarly, we can estimate the error bounds for right regular part in the domain $(x_{i},t_{j+1})\in\Omega^{N+}$:
\begin{equation}
\label{V+}
\lvert (V^{+(j+1)}-v^{+(j+1)})(x_i)\rvert_{\Omega^{N+}} \le C(N^{-2}+\Delta t^{2}).
\end{equation}
Combining the above results \eqref{V-} and \eqref{V+}, we obtain
$$\|V-v\|_{\Omega^{N-}\cup\Omega^{N+}}\le C(N^{-2}+\Delta t^{2}).$$
\end{proof}
\begin{lemma}
	\label{left layer}
	Let $W_{l}(x_{i},t_{j+1}), w_{l}(x,t)$ are solution of the problem \eqref{LW-} and \eqref{lw} respectively. The left singular component of the truncation error satisfies the following estimate for $\sqrt{\alpha}\mu \le \sqrt{\rho\epsilon}$:
 \begin{equation}
     \lvert \|W_{l}-w_{l}\|_{\Omega^{N-}\cup\Omega^{N+}} \le \left\{
	\begin{array}{ll}
	\displaystyle C((N^{-1}\ln N)^2+\Delta t^2), & \hbox{ $\sqrt{\alpha}\mu \le \sqrt{\rho\epsilon}$, } \vspace{.5mm}\\
\displaystyle	C(N^{-2}(\ln N)^3+\Delta t^2), & \hbox{ $\sqrt{\alpha}\mu > \sqrt{\rho\epsilon}$.}
	\end{array}
	\right.
 \end{equation}
\end{lemma}
\begin{proof}
    The truncation error in case of uniform $(\tau_{1}=\tau_{2}=\frac{d}{4})$ for the case $\sqrt{\alpha}\mu\le\sqrt{\rho\epsilon}$ in the domain $(x_{i},t_{j+1})\in\Omega^{N-}$:
\begin{align*}
    \lvert \mathcal{L}_{c}^{N,M}(W_{l}^{-(j+1)}-w_{l}^{-(j+1)})(x_{i})\rvert
&\le CN^{-2}(\epsilon\|w_{l,xxxx}^{-(j+1}\|+\mu\|w_{l,xxx}^{-(j+1)}\|)+C_{1}\Delta t^{2}\| w_{tt}\|\\&\le
\frac{C_{1}h_{1}^{2}}{\epsilon}+C_{2}\Delta t^2\\
&\le C(N^{-2}(\ln N)^{2}+\Delta t^2).
\end{align*}

In case of non-uniform mesh, we split the argument into two parts.
For $(x_{i},t_{j+1})\in [\tau_{1}, d)\times(0,T]$ from Theorem (\ref{bounds1}), we obtain
	\begin{equation}
	\label{wl}
	\lvert w_{l}^{-(j+1)}(x_{i})\rvert \le C\exp^{-\theta_{2}x_{i}}\le C\exp^{- \theta_{2} \tau_{1}}\le CN^{-2}.
	\end{equation}
	Also from Lemma (\ref{discrete singular}),
	$$ \lvert W_{l}^{-(j+1)}(x_{i})\rvert \le C\prod_{k=1}^{\frac{N}{8}}(1+\theta_{2}h_{k})^{-1}\le\bigg(1+\frac{16 \ln N}{N}\bigg)^{-N/8}\le CN^{-2}.$$
 Hence, for all $(x_{i},t_{j+1})\in [\tau_{1},d)\times(0,T]$, we have
	$$\lvert (W_{l}^{-(j+1)}-w_{l}^{-(j+1)})(x_{i})\rvert \le \lvert W_{l}^{-(j+1)}(x_{i})\rvert +\lvert w_{l}^{-(j+1)}(x_{i})\rvert \le CN^{-2}.$$
 For $\sqrt{\alpha} \mu\le\sqrt{\rho\epsilon} $, the  truncation error for the left layer component  in the inner region $(0,\tau_{1})\times(0,T],$ is
 \begin{align*}
\lvert \mathcal{L}_{c}^{N,M}(W_{l}^{-(j+1)}-w_{l}^{-(j+1)})(x_{i})\rvert
&\le \bigg\lvert \epsilon \bigg(\delta^{2}-\frac{\delta^{2}}{\delta x^{2}}\bigg)w_{l}^{-(j+1)}(x_{i})\bigg\rvert +\mu \lvert a^{j+\frac{1}{2}}(x_{i})\rvert  \bigg\lvert \bigg(D^{0}-\frac{\delta}{\delta x}\bigg)w_{l}^{-(j+1)}(x_{i})\bigg\rvert\\&+\bigg\lvert \bigg(D_{t}^{-}w^{-(j+1)}(x_{i})-\frac{\delta}{\delta t}w^{-(j+\frac{1}{2})}(x_{i})\bigg)\bigg\rvert\\&\le CN^{-2}(\epsilon\|w_{l,xxxx}^{-}\|+\mu\|w_{l,xxx}^{-}\|)+C_{1}\Delta t^2\\&\le
\frac{Ch_{1}^{2}}{\epsilon}+C_{1}\Delta t^2\\
&\le C(N^{-2}(\ln N)^{2}+\Delta t^2).
\end{align*}
We choose the barrier function for the layer component as
	$$\psi^{j+1}(x_{i})=C_{1}((N^{-1}\ln N)^{2}+\Delta t^2)\pm(W_{l}^{-(j+1)}-w_{l}^{-(j+1)})(x_{i}),~~(x_{i},t_{j+1})\in\Omega^{N-}.$$
	For sufficiently large $C$, we have $\psi^{j+1}(x_{0})\ge0,\psi^{j+1}(x_{N/2})\ge0,\psi^{0}(x_{i})\ge0$ and $\mathcal{L}_{c}^{N,M}\psi^{j+1}(x_{i})\le0$.  Hence by discrete maximum principle in \cite{RPS}, $\psi^{j+1}(x_{i})\ge0$.
 So, by using discrete maximum principle, we can obtain the following bounds:
	$$\lvert (W_{l}^{-(j+1)}-w_{l}^{-(j+1)})(x_{i})\rvert \le C(N^{-2}(\ln N)^2+\Delta t^2),\,\quad  \forall \; (x_{i},t_{j+1})\in\Omega^{N-}.$$
For the case $\sqrt{\alpha}\mu>\sqrt{\rho\epsilon}$ when the mesh is uniform $(\tau_{1}=\tau_{2}=\frac{d}{4})$ in $(x_{i},t_{j+1})\in\Omega^{N-}$:
\begin{align*}
    \lvert \mathcal{L}_{h}^{N,M}(W_{l}^{-(j+1)}-w_{l}^{-(j+1)})(x_{i})\rvert
\le C(N^{-2}(\ln N)^{3}+\Delta t^2).
\end{align*}
The truncation error for the left layer component  in the inner region $(0,\tau_{1})\times(0,T],$ when mesh is non uniform is
 \begin{align*}
\lvert \mathcal{L}_{c}^{N,M}(W_{l}^{-(j+1)}-w_{l}^{-(j+1)})(x_{i})\rvert&\le CN^{-2}(\epsilon\|w_{l,xxxx}^{-}\|+\mu\|w_{l,xxx}^{-}\|)+C_{1}\Delta t^2\\
&\le CN^{-2}(\ln N)^{3}+C_{1}\Delta t^2.
\end{align*}
\begin{align*}
\lvert \mathcal{L}_{m}^{N,M}(W_{l}^{-(j+1)}-w_{l}^{-(j+1)})(x_{i})\rvert&\le Ch_{1}\epsilon\|w_{l,xxx}^{-}\|+ Ch_{1}^{2}\mu(\|w_{l,xxx}^{-}\|+\|w_{l,xx}^{-}\|)+C_{1}\Delta ^{2}\\
&\le Ch_{1}^{2}\mu(\|w_{l}^{-(3)}\|+\|w_{l}^{-(2)}\|)+C_{1}\Delta t^2\\
&\le C(N^{-2}(\ln N)^{3}+\Delta t^2).
\end{align*}
\begin{align*}
\lvert \mathcal{L}_{u}^{N,M}(W_{l}^{-(j+1)}-w_{l}^{-(j+1)})(x_{i})\rvert&\le CN^{-2}(\epsilon\|w_{l,xxxx}^{-}\|+\mu\|w_{l,xx}^{-}\|)+C_{1}\Delta t^2\\
&\le C(N^{-2}(\ln N)^{3}+\Delta t^2).
\end{align*}
We choose the barrier function for the layer component as
	$$\psi^{j+1}(x_{i})=C_{1}((N^{-1}\ln N)^{3}+\Delta t^2)\pm(W_{l}^{-(j+1)}-w_{l}^{-(j+1)})(x_{i}),\; (x_{i},t_{j+1}))\in \Omega^{N-}.$$
	For sufficiently large $C$, we have $\psi^{j+1}(x_{0})\ge0,\psi^{j+1}(x_{N/2})\ge0,\psi^{0}(x_{i})\ge0$ and $\mathcal{L}_{h}^{N,M}\psi^{j+1}(x_{i})\le0$.  Hence by discrete maximum principle in \cite{RPS}, $\psi^{j+1}(x_{i})\ge0$.
 So, we can obtain the following bounds in $(0,d)\times(0,T]$:
	\begin{equation}
	\label{leftWl-}
	\lvert (W_{l}^{-(j+1)}-w_{l}^{-(j+1)})(x_i)\rvert \le \left\{
	\begin{array}{ll}
	\displaystyle C((N^{-1}\ln N)^2+\Delta t^2), & \hbox{ $\sqrt{\alpha}\mu \le \sqrt{\rho\epsilon}$, } \vspace{.5mm}\\
\displaystyle	C(N^{-2}(\ln N)^3+\Delta t^2), & \hbox{ $\sqrt{\alpha}\mu > \sqrt{\rho\epsilon}$.}
	\end{array}
	\right.
	\end{equation}
 By similar argument in the domain  $(d, 1)\times(0,T]$, we have 
	\begin{equation}
	\label{rightWl+}
	\lvert (W_{l}^{+(j+1)}-w_{l}^{+(j+1)})(x_i)\rvert \le \left\{
	\begin{array}{ll}
	\displaystyle C((N^{-1}\ln N)^2+\Delta t^2), & \hbox{ $\sqrt{\alpha}\mu \le \sqrt{\rho\epsilon}$, } \vspace{.5mm}\\
\displaystyle	C(N^{-2}(\ln N)^3+\Delta t^2), & \hbox{ $\sqrt{\alpha}\mu > \sqrt{\rho\epsilon}$.}
	\end{array}
	\right.
	\end{equation}
	Combining the results (\ref{leftWl-}) and (\ref{rightWl+}), the desired result is obtained.
\end{proof}
\begin{lemma}
	\label{right layer}
	Let $W_{r}^{-}(x_{i},t_{j+1}), w_{r}^{-}(x,t)$ are solution of the problem \eqref{RW-} and \eqref{rw} respectively. The right singular component of the truncation error satisfies the following estimate for $\sqrt{\alpha}\mu \le \sqrt{\rho\epsilon}$:
 \begin{equation}
     \|W_{r}-w_{r}\|_{\Omega^{N-}\cup\Omega^{N+}} \le \left\{
	\begin{array}{ll}
	\displaystyle C((N^{-1}\ln N)^2+\Delta t^2), & \hbox{ $\sqrt{\alpha}\mu \le \sqrt{\rho\epsilon}$, } \vspace{.5mm}\\
\displaystyle	C(N^{-2}(\ln N)^3+\Delta t^2), & \hbox{ $\sqrt{\alpha}\mu > \sqrt{\rho\epsilon}$.}
	\end{array}
	\right.
 \end{equation}
\end{lemma}
\begin{proof}
  The truncation error for $\sqrt{\alpha}\mu\le\sqrt{\rho\epsilon}$, when the mesh is uniform $(\tau_{1}=\tau_{2}=\frac{d}{4})$ in $(x_{i},t_{j+1})\in\Omega^{N-}$:
\begin{align*}
    \lvert \mathcal{L}_{c}^{N,M}(W_{r}^{-(j+1)}-w_{r}^{-(j+1)})(x_{i})\rvert
&\le CN^{-2}(\epsilon\|w_{l,xxxx}^{-}\|+\mu\|w_{r,xxx}^{-}\|)+C_{1}\Delta t^2\\&\le
\frac{Ch^{2}}{\epsilon}+C_{1}\Delta t^2\\
&\le C(N^{-2}(\ln N)^{2}+\Delta t^2).
\end{align*}
In case of non-uniform mesh, we split the interval $(0,d)\times(0,T]$ into two parts $(0,d-\tau_{2}]\times(0,T]$ and $(d-\tau_{2},d)\times(0,T]$. \\
In the domain $(x_{i},t_{j+1})\in(0,d-\tau_{2}]\times(0,T]$:
\begin{equation}
	\label{rightlayerlemma}
	\lvert w_{r}^{-(j+1)}(x_{i})\rvert \le C\exp^{-\theta_{1}(d-x_{i})}\le C\exp^{- \theta_{1} \tau_{2}}\le CN^{-2}.
	\end{equation}
	Also from Lemma (\ref{discrete singular}),
	$$ \lvert W_{r}^{-(j+1)}(x_{i})\rvert \le C\prod_{k=1}^{\frac{N}{8}}(1+\theta_{1}h_{k})^{-1}\le\bigg(1+\frac{16 \ln N}{N}\bigg)^{-N/8}\le CN^{-2}.$$
 Hence, for all $(x_{i},t_{j+1})\in (0,d-\tau_{2}]\times(0,T]$  we have
	$$\lvert (W_{r}^{-(j+1)}-w_{r}^{-(j+1)})(x_{i})\rvert \le \lvert W_{r}^{-(j+1)}(x_{i})\rvert +\lvert w_{r}^{-(j+1)}(x_{i})\rvert \le CN^{-2}.$$
 For $\sqrt{\alpha} \mu\le\sqrt{\rho\epsilon} $, the  truncation error for the left layer component  in the inner region $(d-\tau_{2},d)\times(0,T],$ is
 \begin{align*}
\lvert \mathcal{L}_{c}^{N,M}(W_{r}^{-(j+1)}-w_{r}^{-(j+1)})(x_{i})\rvert
&\le \bigg\lvert \epsilon \bigg(\delta^{2}-\frac{d^{2}}{\delta x^{2}}\bigg)w_{r}^{-(j+1)}(x_{i})\bigg\rvert +\mu \lvert a^{j+\frac{1}{2}}(x_{i})\rvert  \bigg\lvert \bigg(D^{0}-\frac{\delta}{\delta x}\bigg)w_{r}^{-(j+1)}(x_{i})\bigg\rvert\\&+\bigg\lvert \bigg(D_{t}^{-}w_{r}^{-(j+1)}(x_{i})-\frac{\delta}{\delta t}w_{r}^{-(j+\frac{1}{2})}(x_{i})\bigg)\bigg\rvert\\&\le CN^{-2}(\epsilon\|w_{r,xxxx}^{-}\|+\mu\|w_{r,xxx}^{-}\|)+C_{1}\Delta t^{2}\\&\le
\frac{Ch_{3}^{2}}{\epsilon}+C_{1}\Delta t^2\\
&\le C(N^{-2}(\ln N)^{2}+\Delta t^2).
\end{align*}
We choose the barrier function for the layer component as
	$$\psi^{j+1}(x_{i})=C_{1}((N^{-1}\ln N)^{2}+\Delta t^2)\pm(W_{r}^{-(j+1)}-w_{r}^{-(j+1)})(x_{i}),~~(x_{i},t_{j+1})\in\Omega^{N-}.$$
	For sufficiently large $C$, we have $\psi^{j+1}(x_{0})\ge0,\psi^{j+1}(x_{N/2})\ge0,\psi^{0}(x_{i})\ge0$ and $\mathcal{L}_{h}^{N,M}\psi^{j+1}(x_{i})\le0$.  Hence by discrete maximum principle in \cite{RPS}, $\psi^{j+1}(x_{i})\ge0$.
 So, by using discrete maximum principle, we can obtain the following bounds:
	$$\lvert (W_{r}^{-(j+1)}-w_{r}^{-(j+1)})(x_{i})\rvert \le C(N^{-2}(\ln N)^2+\Delta t^2),\,\quad  \forall \; (x_{i},t_{j+1})\in\Omega^{N-}.$$
For $\sqrt{\alpha}\mu>\sqrt{\rho\epsilon}$, when mesh is uniform $(\tau_{1}=\tau_{2}=\frac{d}{4})$ in $(0,d)\times(0,T]$:
\begin{align*}
    \lvert \mathcal{L}_{h}^{N,M}(W_{r}^{-(j+1)}-w_{r}^{-(j+1)})(x_{i})\rvert
&\le CN^{-2}(\epsilon\|w_{r,xxxx}^{-}\|+\mu\|w_{r,xxx}^{-}\|)+C_{1}\Delta t^2\\
&\le C(N^{-2}(\ln N)^{3}+\Delta t^2).
\end{align*}
In case of non-uniform mesh, the truncation error for the left layer component  in the inner region $(d-\tau_{2},d)\times(0,T],$ is
 \begin{align*}
\lvert \mathcal{L}_{c}^{N,M}(W_{r}^{-(j+1)}-w_{r}^{-(j+1)})(x_{i})\rvert&\le CN^{-2}(\epsilon\|w_{r,xxxx}^{-}\|+\mu\|w_{r,xxx}^{-}\|)+C_{1}\Delta t^2\\
&\le CN^{-2}(\ln N)^{2}+C_{1}\Delta t^2.
\end{align*}
We choose the barrier function for the layer component as
	$$\psi^{j+1}(x_{i})=C_{1}(N^{-2}\ln N^{3}+\Delta t^2)\pm(W_{r}^{-(j+1)}-w_{r}^{-(j+1)})(x_{i}),~~(x_{i},t_{j+1})\in\Omega^{N-}.$$
	For sufficiently large $C$, we have $\psi^{j+1}(x_{0})\ge0,\psi^{j+1}(x_{N/2})\ge0,\psi^{0}(x_{i})\ge0$ and $\mathcal{L}_{h}^{N,M}\psi^{j+1}(x_{i})\le0$.  Hence by discrete maximum principle in \cite{RPS}, $\psi^{j+1}(x_{i})\ge0$.
 So, by using discrete maximum principle, we can obtain the following bounds in $(x_{i},t_{j+1})\in\Omega^{N-}$:
	\begin{equation}
	\label{leftWr-}
	\lvert (W_{r}^{-(j+1)}-w_{r}^{-(j+1)})(x_i)\rvert \le \left\{
	\begin{array}{ll}
	\displaystyle C((N^{-1}\ln N)^2+\Delta t^2), & \hbox{ $\sqrt{\alpha}\mu \le \sqrt{\rho\epsilon}$, } \vspace{.5mm}\\
\displaystyle	C(N^{-2}(\ln N)^3+\Delta t^2), & \hbox{ $\sqrt{\alpha}\mu > \sqrt{\rho\epsilon}$.}
	\end{array}
	\right.
	\end{equation}
 By similar argument in the domains  $(x_{i},t_{j+1})\in\Omega^{N+}$, we have 
	\begin{equation}
	\label{rightWr+}
	\lvert (W_{r}^{+(j+1)}-w_{r}^{+(j+1)})(x_i)\rvert \le \left\{
	\begin{array}{ll}
	\displaystyle C((N^{-1}\ln N)^2+\Delta t^2), & \hbox{ $\sqrt{\alpha}\mu \le \sqrt{\rho\epsilon}$, } \vspace{.5mm}\\
\displaystyle	C(N^{-2}(\ln N)^3+\Delta t^2), & \hbox{ $\sqrt{\alpha}\mu > \sqrt{\rho\epsilon}$.}
	\end{array}
	\right.
	\end{equation}
	Combining the results (\ref{leftWr-}) and (\ref{rightWr+}), the desired result is obtained.
\end{proof}

\begin{lemma}
		The error $e(x_{\frac{N}{2}},t_{j+1})$ estimated at the point of discontinuity $(x_{\frac{N}{2}},t_{j+1})=(d,t_{j+1}),0\le j\le M-1$ satisfies the following estimates:
		$$\displaystyle \lvert \mathcal{L}_{h}^{N,M}(U-u)(x_{\frac{N}{2}},t_{j+1})\rvert \le \left\{
	\begin{array}{ll}
	\displaystyle \frac{C h^{2}}{\epsilon^{3/2}}, & \hbox{ $\sqrt{\alpha}\mu \le \sqrt{\rho\epsilon}$, } \vspace{.5mm}\\
 \displaystyle	\frac{C h^{2}\mu^3}{\epsilon^{3}}, & \hbox{ $\sqrt{\alpha}\mu > \sqrt{\rho\epsilon}$,}
	\end{array}
	\right. $$
 where $h=\max\{h_{3},h_{4}\}$.
\end{lemma}
\begin{proof}
    For $\sqrt{\alpha} \mu\le\sqrt{\rho\epsilon} $ the  truncation error at the point $x_{\frac{N}{2}}=d$:
\begin{align*}
	\bigg\lvert \mathcal{L}_{h}^{N,M} U^{j+1}(x_{N/2})-\frac{2h_{3}^{2}g^{j+1}_{N/2-1}h_{4}}{2\epsilon-h_{3}\mu a^{j+\frac{1}{2}}_{N/2-1}}-\frac{2h_{4}^{2}g^{j+1}_{N/2+1}h_{3}}{2\epsilon+h_{4}\mu a^{j+\frac{1}{2}}_{N/2+1}}\bigg\rvert &\le\lvert\mathcal{L}_{t}^{N,M} U^{j+1}(x_{N/2})\rvert+\frac{2h_{3}^{2} h_{4}}{2\epsilon-h_{3}\mu a^{j+\frac{1}{2}}_{N/2-1}}\\&\lvert\mathcal{L}_{c}^{N,M} u(x_{N/2-1},t_{j+1})-g^{j+1}_{N/2-1}\rvert +\frac{2h_{4}^{2}h_{3}}{2\epsilon+h_{4}\mu a^{j+\frac{1}{2}}_{N/2+1}}\\&\lvert\mathcal{L}_{c}^{N,M} u(x_{N/2+1},t_{j+1})-g^{j+1}_{N/2+1}\rvert\\&\le Ch^{2}\|u_{xxx}\| \\&\le\frac{Ch^2}{\epsilon^{3/2}}
	\end{align*}
 For $\sqrt{\alpha} \mu>\sqrt{\rho\epsilon} $, the  truncation error at the point $x_{\frac{N}{2}}=\frac{d}{2}$:
 \begin{align*}
	\bigg\lvert \mathcal{L}_{h}^{N,M} U^{j+1}(x_{N/2})-\frac{2h_{3}^{2}g^{j+\frac{1}{2}}_{N/2-1}h_{4}}{2\epsilon-h_{3}\mu a^{j+\frac{1}{2}}_{N/2-1}}-\frac{2h_{4}^{2}g^{j+\frac{1}{2}}_{N/2+1}h_{3}}{2\epsilon+h_{4}\mu a^{j+\frac{1}{2}}_{N/2+1}}\bigg\rvert &\le\lvert\mathcal{L}_{t}^{N,M} U^{j+1}(x_{N/2})\rvert+\frac{2h_{3}^{2} h_{4}}{2\epsilon-h_{3}\mu a^{j+\frac{1}{2}}_{N/2-1}}\\&\lvert \mathcal{L}_{c}^{N,M} u(x_{N/2-1},t_{j+1})-g^{j+1}_{N/2-1}\rvert +\frac{2h_{4}^{2}h_{3}}{2\epsilon+h_{4}\mu a^{j+\frac{1}{2}}_{N/2+1}}\\&\lvert\mathcal{L}_{c}^{N,M} u(x_{N/2+1},t_{j+1})-g^{j+1}_{N/2+1}\rvert\\&\le Ch^{2}\|u_{xxx}\|  \\&\le\frac{Ch^2\mu^3}{\epsilon^3}
	\end{align*}

\end{proof}
\begin{theorem}
	\label{main}
    Let $u(x,t)$ and $ U(x,t)$ be the solutions of \eqref{twoparaparabolic} and \eqref{discrete problem} respectively, then
	$$\|U-u\|_{\bar{\Omega}^{N,M}}\le
	C(N^{-2}\ln N^3+\Delta t^{2})$$
	where C is a constant independent of $\epsilon, \mu$ and discretization parameter $N,M$. 
\end{theorem}
\begin{proof}
From Lemma (\ref{regular part}), Lemma (\ref{left layer}), and Lemma (\ref{right layer}), we have that
	 	$$\displaystyle \|U-u\|_{\Omega^{N-}\cup\Omega^{N+}} \le \left\{
	\begin{array}{ll}
	\displaystyle C(N^{-2} (\ln N)^{2}+\Delta t^2), & \hbox{ $\sqrt{\alpha}\mu \le \sqrt{\rho\epsilon}$, } \vspace{.5mm}\\
\displaystyle	C(N^{-2} (\ln N)^{3}+\Delta t^2), & \hbox{ $\sqrt{\alpha}\mu > \sqrt{\rho\epsilon}$.}
	\end{array}
	\right. $$
  Let $\sqrt{\alpha}\mu \le \sqrt{\rho\epsilon},$ to find error at the point of discontinuity $x_{\frac{N}{2}}=d$, we have considered the discrete barrier function $\phi_{1}(x_{i},t_{j+1})=\psi_{1}(x_{i},t_{j+1})\pm e(x_{i},t_{j+1})$ defined in the interval $(d-\tau_{2},d+\tau_{3})\times(0,T]$, where
	 $$\psi_{1}(x_{i},t_{j+1})=C(N^{-2}(\ln N)^{3}+\Delta t^2)+\frac{C h^2}{\epsilon^{3/2}}\left\{
	 \begin{array}{ll}
	 	\displaystyle x_{i}-(d+\tau_{2}), & \hbox{ $(x_{i},t_{j+1})\in \Omega^{N,M}\cap(d-\tau_{2}, d)\times(0,T]$, }\\
	 	d-\tau_{3}-x_{i}, & \hbox{ $(x_{i},t_{j+1})\in\Omega^{N,M}\cap(d, d+\tau_{3})\times(0,T].$}
	 \end{array}
	 \right. $$
	 For $ x_{i}\in(d-\tau_{2},d+\tau_{3}), \; \phi_{1}(d-\tau_{2},t_{j+1})$, $\phi_{1}(d+\tau_{3},t_{j+1}),\phi_{1}(x_{i},t_{0})$ are non-negative and
	 $\mathcal{L}_{h}^{N,M}\phi_{1}(x_{i},t_{j+1})\le 0,~~\forall(x_{i},t_{j+1})\in( d-\sigma_2, d+\sigma_3)\times(0,T].$ Also~$ \mathcal{L}_{t}^{N,M}\phi_{1}(x_{N/2},t_{j+1}) \le0.$\\
  Hence, by applying discrete minimum principle we get $\phi_{1}(x_{i},t_{j+1})\ge0.$\\
	Therefore, for $(x_{i},t_{j+1})\in (d-\tau_{2},d+\tau_{3})\times(0,T]:$
	\begin{equation}
	\label{Yd}
	\lvert (U-u)(x_{i},t_{j+1})\rvert \le C_{1}(N^{-2}(\ln N)^3+\Delta t^2)+\frac{C_{2}h^{2}\sigma}{\epsilon^{3/2}}\le C(N^{-2}(\ln N)^3+\Delta t^2).
	\end{equation}
 In second case $\sqrt{\alpha}\mu > \sqrt{\rho\epsilon}$, consider the discrete barrier function $\phi_{2}(x_{i},t_{j+1})=\psi_{2}(x_{i},t_{j+1})\pm e(x_{i},t_{j+1})$ defined in the interval $(d-\tau_{2},d+\tau_{3})\times(0,T]$, where
	$$\psi_{2}(x_{i},t_{j+1})=C(N^{-2}(\ln N)^{3}+\Delta t^2)+C_{1}\frac{h^2\mu^3}{\epsilon^3}\left\{
	\begin{array}{ll}
	\displaystyle x_{i}-(d+\tau_{2}), & \hbox{ $(x_{i},t_{j+1})\in \Omega^{N,M}\cap(d-\tau_{2}, d)\times(0,T] $, }\\
	d-\tau_{3}-x_{i}, & \hbox{ $(x_{i},t_{j+1})\in\Omega^{N,M}\cap(d, d+\tau_{3})\times(0,T].$}
	\end{array}
	\right.$$
	We have $\phi_{2}(d-\tau_{2},t_{j+1})$ and $\phi_{2}(d+\tau_{3},t_{j+1}),\phi_{2}(x_{i},t_{0}),~x_{i}\in(d-\tau_{2},d+\tau_{3}) $ are non negative and
	 $\mathcal{L}_{h}^{N,M}\phi_{2}(x_{i},t_{j+1})\le 0,~~\forall(x_{i},t_{j+1})\in(d-\tau_{2},d+\tau_{3})\times(0,T]$.~ Also $\mathcal{L}_{t}^{N,M}\phi_{2}(x_{i},t_{j+1})\le 0.$\\
	Hence by applying discrete minimum principle, we get $\phi_{2}(x_{i},t_{j+1})\ge0.$
	Therefore, for $(x_{i},t_{j+1})\in (d-\tau_{2},d+\tau_{3})\times(0,T]$
	\begin{equation}
	\label{Yd1}
	\lvert (U-u)(x_{i},t_{j+1})\rvert \le C(N^{-2}(\ln N)^3+\Delta t^2).
	\end{equation}
	By combining the result (\ref{Yd}) and (\ref{Yd1}) we obtain the desired result.
\end{proof}
\section{Numerical examples} 
To demonstrate the effectiveness of the proposed hybrid difference approach, we proposed numerical method  on two test problems with discontinuous convection coefficients and source terms. As the exact solutions to these problems are unknown, we used the double mesh method \cite{DV,D} to evaluate the accuracy of the numerical approximations obtained. The double mesh difference is defined by
$$E^{N,M}= \max\limits_{j}\bigg(\max\limits_{i}|U_{2i,2j}^{2N,2M}-U_{i,j}^{N,M}|\bigg)$$
where $U_{i}^{N,M}$ and $U_{2i}^{2N,2M}$ are the solutions on the mesh $\bar{\Omega}^{N,M}$ and $\bar{\Omega}^{2N,2M}$ respectively. 
The order of convergence is given by
$$R^{N,M}=\log_{2}\bigg(\frac{E^{N,M}}{E_{\epsilon,\mu}^{2N,2M}}\bigg).$$
\begin{Example}\label{ex-a}
	$$(\epsilon u_{xx}+\mu au_{x}-bu-u_{t})(x,t)=f(x,t),~~ ~~(x,t)\in((0,.5)\cup(0.5,1))\times(0,1],$$
	$$	u(0,t)=u(1,t)=u(x,0)=0,$$\nonumber\
	with \\
	$$a(x,t)=\left\{
	\begin{array}{ll}
	\displaystyle -(1+x(1-x)), & \hbox{ $0\le x\le 0.5,t\in(0,1]$, }\\
	1+x(1-x), & \hbox{ $0.5<x\le1,t\in(0,1]$,}
	\end{array}
	\right.$$
	$$
	f(x,t)=\left\{
	\begin{array}{ll}
	\displaystyle -2(1+x^2)t, & \hbox{ $0\le x\le 0.5,t\in(0,1]$, }\\
	2(1+x^2)t, & \hbox{ $0.5<x\le1,t\in(0,1]$,}
	\end{array}
	\right.$$\\
	and $b(x,t)=1+exp(x)$.
\end{Example}
\begin{Example}\label{ex-b}
	\begin{eqnarray*}
 (\epsilon u_{xx}&+&\mu au_{x}-bu-u_{t})(x,t)=f(x,t),~~ ~~(x,t)\in ((0,0.5))\cup (0.5,1))\times(0,1],\\
	u(0,t)&=&u(1,t)=u(x,0)=0,
 \end{eqnarray*}
	with
 \[a(x,t)=\left\{
	\begin{array}{ll}
	\displaystyle -(1+x(1-x)), & 0\le x\le 0.5,t\in(0,1], \\
	1+x(1-x), & 0.5<x\le1,t\in(0,1],
	\end{array}
	\right. \]
	\[
	f(x,t)=\left\{
	\begin{array}{ll}
	\displaystyle -2(1+x^2)t, &  0\le x\le0.5,t\in(0,1], \\
	3(1+x^2)t, & 0.5<x\le1,t\in(0,1],
	\end{array}
	\right.\] with 
	and $b(x,t)=1+\exp(x)$.
\end{Example}
\begin{Example}\label{ex-c}
	$$(\epsilon u_{xx}+\mu au_{x}-bu-u_{t})(x,t)=f(x,t),~~ ~~(x,t)\in((0,.5)\cup(0.5,1))\times(0,1],$$
	$$	u(0,t)=u(1,t)=u(x,0)=0,$$\nonumber\
	with \\
	$$a(x,t)=\left\{
	\begin{array}{ll}
	\displaystyle -(1+\exp(-xt)), & \hbox{ $0\le x\le 0.5,t\in(0,1]$, }\\
	2+x+t, & \hbox{ $0.5<x\le1,t\in(0,1]$,}
	\end{array}
	\right. $$	
	$$
	f(x,t)=\left\{
	\begin{array}{ll}
	\displaystyle (\exp(t^{2})-1)(1+xt), & \hbox{ $0\le x\le 0.5,t\in(0,1]$, }\\
	-(2+x)t^{2}, & \hbox{ $0.5<x\le1,t\in(0,1]$,}
	\end{array}
	\right.$$\\
	and $b(x,t)=2+xt$.
\end{Example}
%\section{Tables}\label{sec7}
\begin{table}
\small
\begin{center}
\caption{Maximum point-wise error $E^{N,M}$ and approximate orders of convergence $R^{N,M}$ for Example \ref{ex-a} when $\epsilon=2^{-6}.$}\label{eg11}%
\begin{tabular}{ccccccc}
\hline
\multirow{2}{*}{$\mu$ }& \multicolumn{6}{c}{Number of mesh points N}\\
\cline{2-7}
 & 32   & 64 &128 & 256&512&1024\\
\hline
		$2^{-8}$ & 4.17e-03&1.41e-03&	6.13e-04&	2.94e-04&	1.45e-04&6.99e-05 \\ 
		Order &1.5623&	1.2055&	1.0627&	1.0181&1.0526& \\

		$2^{-10}$ &3.32e-03&	9.02e-04	&2.45e-04&	8.45e-05&	3.88e-05& 1.93e-05\\ 
		Order & 1.8769& 1.8775&	1.5374&	1.1253&1.0074& \\
		$2^{-12}$ &3.24e-03	&8.65e-04&	2.26e-04&	5.92e-05&	1.58e-05 &3.98e-06\\ 
		Order & 1.9038&	1.9330&	1.9357&	1.9099&1.9891 \\
  
		$2^{-14}$ & 3.22e-03&	8.56e-04	&2.22e-04&	5.68e-05&	1.45e-05&3.66e-06\\ 
		
		Order & 1.9108&	1.9481&	1.9657&	1.9672&1.9861& \\
		
		$2^{-16}$ &3.22e-03&	8.54e-04	&2.21e-04&	5.62e-05	&1.42e-05&3.55e-06\\
		Order & 1.9902&	1.9972&	1.9991&	1.9995&1.9991 \\
		$2^{-18}$ &3.21e-03&	8.53e-04&	2.20e-04	&5.60e-05&	1.41e-05&3.55e-06\\
		Order & 1.9131&	1.9529&	1.9755&	1.9867&1.9898 \\
		$2^{-20}$ &3.21e-03&	8.53e-04&	2.20e-04	&5.60e-05&	1.41e-05&3.53e-06\\
		Order & 1.9132& 1.9532&	1.9753&	1.9867&1.9979 \\
  $2^{-22}$ &3.21e-03&	8.53e-04	&2.20e-04	&5.60e-05&	1.41e-05&3.53e-06\\
		Order & 1.9132&	1.9532&	1.9761&	1.9879& 1.9979\\
  $2^{-24}$ &3.21e-03&	8.53e-04	&2.20e-04&	5.60e-05&	1.41e-05&3.53e-06\\
		Order & 1.9132&	1.9532&	1.9761&	1.9880&1.9979 \\
  \vdots &\vdots&\vdots&\vdots&\vdots&\vdots&\vdots\\
  $2^{-60}$ &3.21e-03&	8.53e-04&	2.20e-04	&5.60e-05	&1.41e-05&3.53e-06\\
		Order & 1.9132	&1.9532&	1.9761&	1.9880&1.9979 \\
\hline
\end{tabular}
\end{center}
\end{table}
\begin{table}
\small
\begin{center}
\caption{Maximum point-wise error $E^{N,M}$ and approximate orders of convergence $R^{N,M}$ for Example \ref{ex-a} when $\mu=2^{-6}.$}\label{eg12}%
\begin{tabular}{cccccc}
\hline
\multirow{2}{*}{$\epsilon$ }& \multicolumn{5}{c}{Number of mesh points N}\\
\cline{2-6}
  & 64 &128 & 256&512&1024\\
\hline 
		$2^{-16}$ & 2.32e-01&	9.41e-02	&2.52e-02	&5.63e-03&	1.67e-03 \\ 
		Order &1.3012&	1.9025&	2.1589&	1.7513& \\

		%$2^{-18}$ &2.3e-01&	1.01e-01&	2.48e-02	&5.58e-03&	1.64e-03 \\ 
		%Order & 1.2027&	2.0226&	2.1498&	1.7715& \\
		$2^{-20}$ & 2.31e-01&	1.03e-01&	2.43e-02&	5.57e-03&	1.46e-03\\ 
		Order & 1.1738&	2.0757&	2.0678&	1.9272& \\
  
		%$2^{-22}$ &2.31e-01&	1.03e-01&	2.42e-02&	5.55e-03&	1.44e-03\\ 
		
		%Order & 1.1659&	2.0901&	2.1239&	1.9509& \\
		
		$2^{-24}$ &2.31e-01&	1.03e-01&2.41e-02&	5.48e-03&	1.35e-03\\
		Order & 1.1643& 2.0955&	2.1374&	2.0222& \\
		%$2^{-26}$ &2.31e-01&	1.03e-01&	2.41e-02&	5.47e-03&	1.34e-03\\
		%Order &1.1639&	2.0965&	2.1412&	2.0303& \\
		$2^{-28}$ &2.31e-01&	1.03e-01&	2.41e-02&	5.47e-03&	1.34e-03\\
		Order & 1.1638&	2.0967&	2.1415&	2.0329& \\
  $2^{-32}$ &2.31e-01&	1.03e-01&	 2.41e-02& 5.46e-03&  	1.33e-03\\
		Order & 1.1637&	2.0969&	2.1424&	2.0330& \\
  $2^{-36}$ &2.31e-01&1.03e-01& 2.41e-02&	5.46e-03&  	1.33e-3\\
		Order & 1.1637&	2.0969&	2.1424&	2.0331& \\
  $2^{-40}$ &2.31e-01&	1.03e-01&	2.41e-02&	5.47e-03&	1.34e-03\\
		Order & 1.1637&	2.0969&	2.1424&	2.0331& \\
\hline
\end{tabular}
\end{center}
\end{table}
\begin{table}
\small
\begin{center}
\caption{Maximum point-wise error $E^{N,M}$ and approximate orders of convergence $R^{N,M}$ for Example \ref{ex-b} when $\mu=2^{-6}.$}\label{eg14}%
\begin{tabular}{cccccc}
\hline
\multirow{2}{*}{$\epsilon$ }& \multicolumn{5}{c}{Number of mesh points N}\\
\cline{2-6}
  & 64 &128 & 256&512&1024\\
\hline 
	$2^{-18}$ & 3.31e-01&1.51e-01&	3.72e-02&9.95e-03&	2.66e-03\\ 
		Order & 1.1321& 2.0226&	1.9015&	1.9011& \\
  
	$2^{-22}$ & 3.55e-01&	1.55e-01&	3.63e-02&	9.75e-03& 2.55e-03\\ 
		Order & 1.1982&	2.0914&	1.8940&	1.9330& \\
$2^{-26}$ & 3.55e-01&	1.55e-01&	3.63e-02&	9.75e-03& 2.55e-03\\ 
		Order & 1.1982&	2.0914&	1.8940&	1.9330& \\
  $2^{-30}$ & 3.55e-01&	1.55e-01&	3.63e-02&	9.75e-03& 2.55e-03\\ 
		Order & 1.1982&	2.0914&	1.8940&	1.9330& \\
  $2^{-34}$ & 3.55e-01&	1.55e-01&	3.63e-02&	9.75e-03& 2.55e-03\\ 
		Order & 1.1982&	2.0914&	1.8940&	1.9330& \\
  $2^{-38}$ & 3.55e-01&	1.55e-01&	3.63e-02&	9.75e-03& 2.55e-03\\ 
		Order & 1.1982&	2.0914&	1.8940&	1.9330& \\ 
  $2^{-42}$ & 3.55e-01&	1.55e-01&	3.63e-02&	9.75e-03& 2.55e-03\\ 
		Order & 1.1982&	2.0914&	1.8940&	1.9330& \\

\hline
\end{tabular}
\end{center}
\end{table}
\begin{table}
\small
\begin{center}
\caption{Maximum point-wise error $E^{N,M}$ and approximate orders of convergence $R^{N,M}$ for Example \ref{ex-c} when $\epsilon=2^{-8}.$}\label{eg13}%
\begin{tabular}{ccccccc}
\hline
\multirow{2}{*}{$\mu$ }& \multicolumn{6}{c}{Number of mesh points N}\\
\cline{2-7}
 & 32   & 64 &128 & 256&512&1024\\
\hline
		$2^{-14}$ &7.74e-03& 2.08e-03&	5.31e-04&	1.34e-04	&4.79e-05	&2.19e-05 \\ 
		Order & 1.8947&	1.9704&	1.9834&	1.4871&	1.1292& \\
		$2^{-16}$ &7.83e-03&2.14e-03&	5.62e-04&	1.34e-04&	3.36e-05&8.49e-06	 \\ 
		Order &1.8692&	1.9572&	1.9654&	1.9974&1.9858 \\
  
		$2^{-20}$ & 7.86e-03&	2.17e-03&	5.69e-04&	1.45e-04	&3.85e-05	&9.72e-06\\ 
		
		Order & 1.9108&	1.9481&	1.9657&	1.9672&1.9868& \\
		
		$2^{-24}$ &7.86e-03&	2.17e-03&	5.69e-04&	1.45e-04&	3.85e-05&	9.72e-06\\
		Order & 1.9902&	1.9972&	1.9991&	1.9995&1.9872& \\
  $2^{-28}$ &7.86e-03&	2.17e-03&	5.69e-04&	1.46e-04&	3.73e-05&	9.43e-06\\
		Order & 1.8592&	1.9299&	1.9605&	1.9694&	1.9845& \\
  $2^{-32}$ &7.86e-03&	2.17e-03&	5.69e-04&	1.46e-04&	3.73e-05&	9.43e-06\\
		Order & 1.8592&	1.9299&	1.9605&	1.9694&	1.9845& \\
  \vdots &\vdots&\vdots&\vdots&\vdots&\vdots&\vdots\\
  $2^{-60}$ &7.86e-03&	2.1e-03&	5.69e-04&	1.46e-04	&3.73e-05&	9.43e-06\\
		Order & 1.8592&	1.9299&	1.9605&	1.9694&	1.9845& \\
\hline
\end{tabular}
\end{center}
\end{table}
\begin{figure}[h]
\centering
\includegraphics[width=0.6\textwidth]{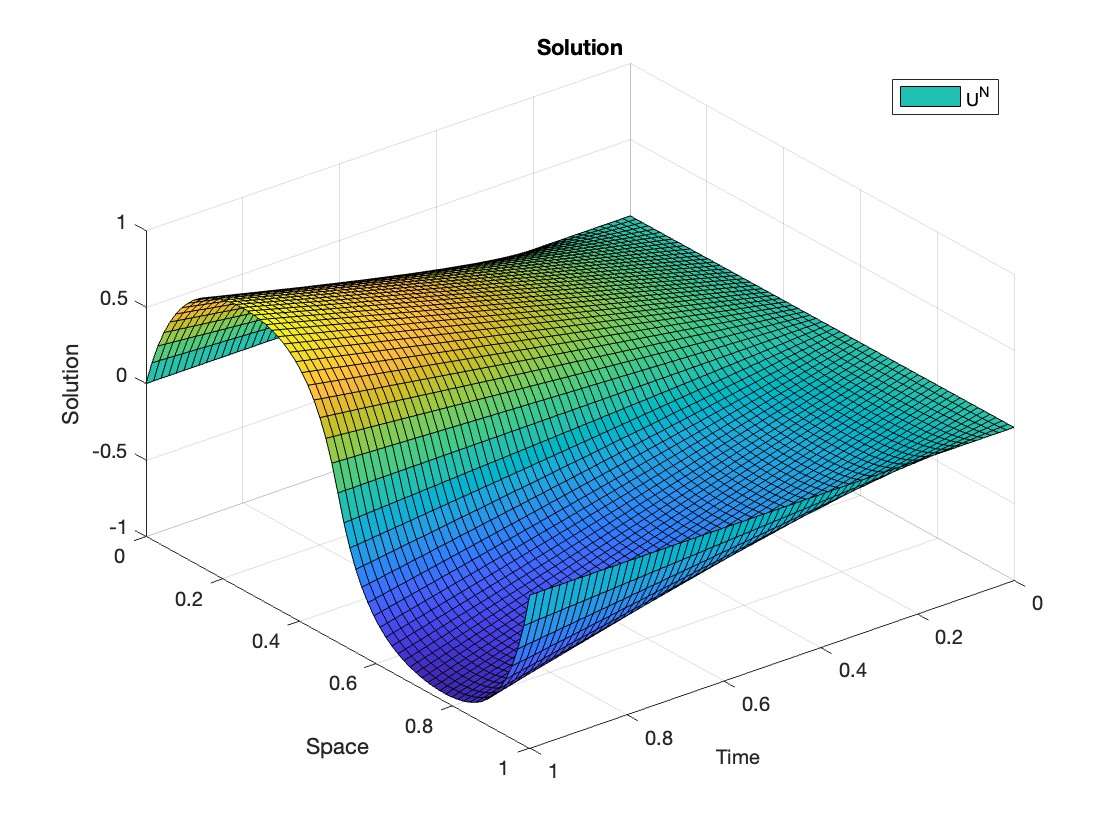}
\caption{Numerical solution for $\epsilon=2^{-6}, \mu=2^{-32}$} when $N=64$ for Example \ref{ex-a}.	
\label{fig11}
\end{figure}
\begin{figure}[h]
\centering
\includegraphics[width=0.6\textwidth]{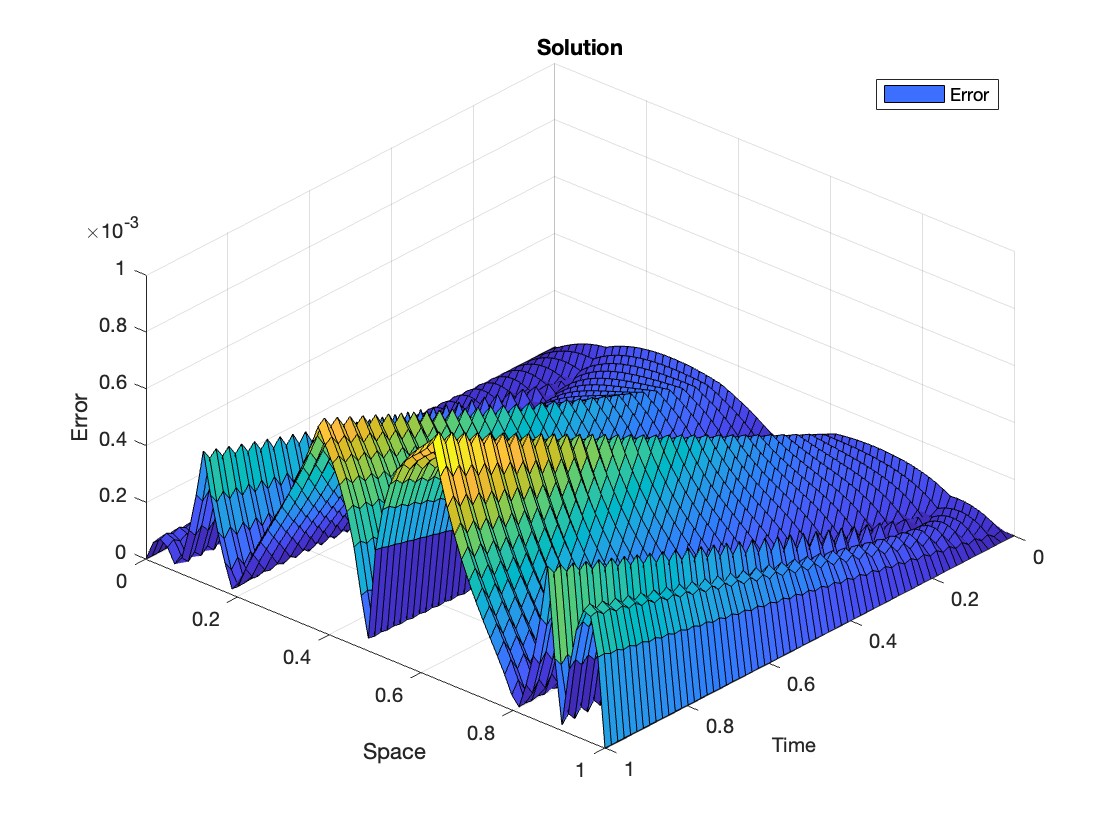}
\caption{Error for $\epsilon=2^{-6}, \mu=2^{-32}$} when $N=64$ for Example \ref{ex-a}.	
\label{fig12}
\end{figure}

\begin{figure}[h]
\centering
\includegraphics[width=0.6\textwidth]{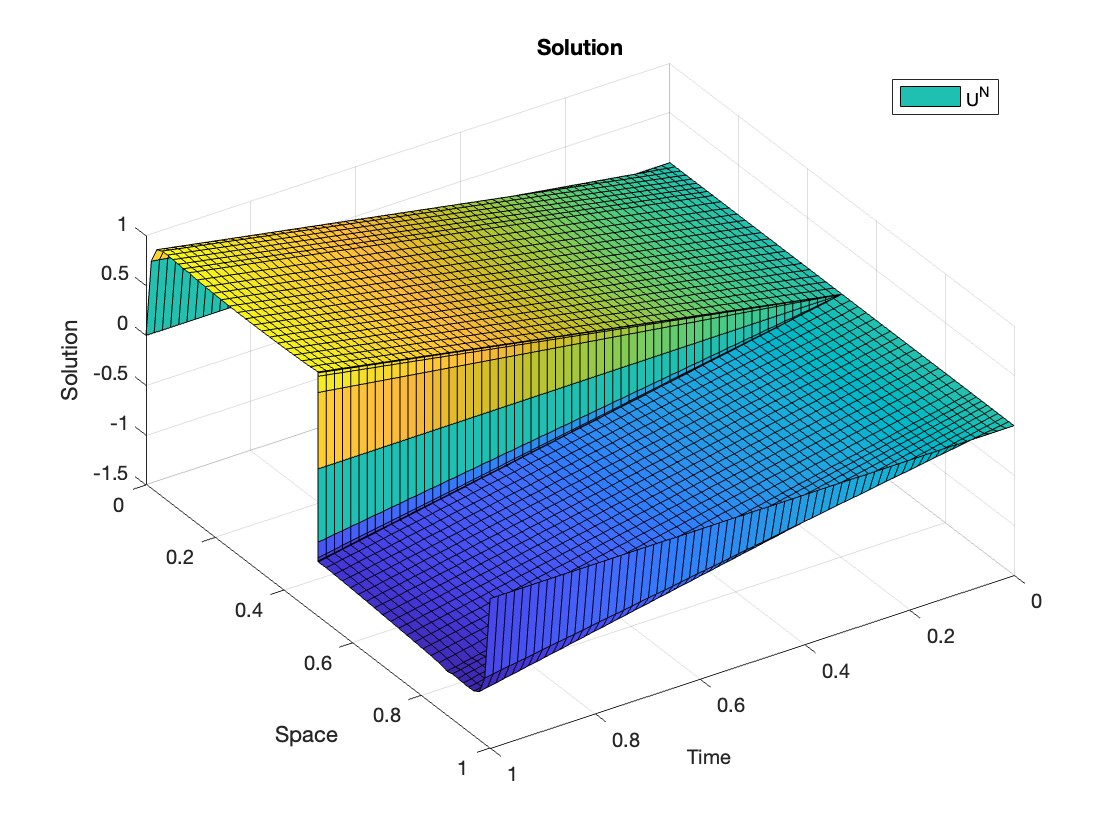}
\caption{Numerical solution for $\epsilon=2^{-30}, \mu=2^{-6}$} when $N=64$ for Example \ref{ex-a}.	
\label{fig22}
\end{figure}
\begin{figure}[h]
\centering
\includegraphics[width=0.6\textwidth]{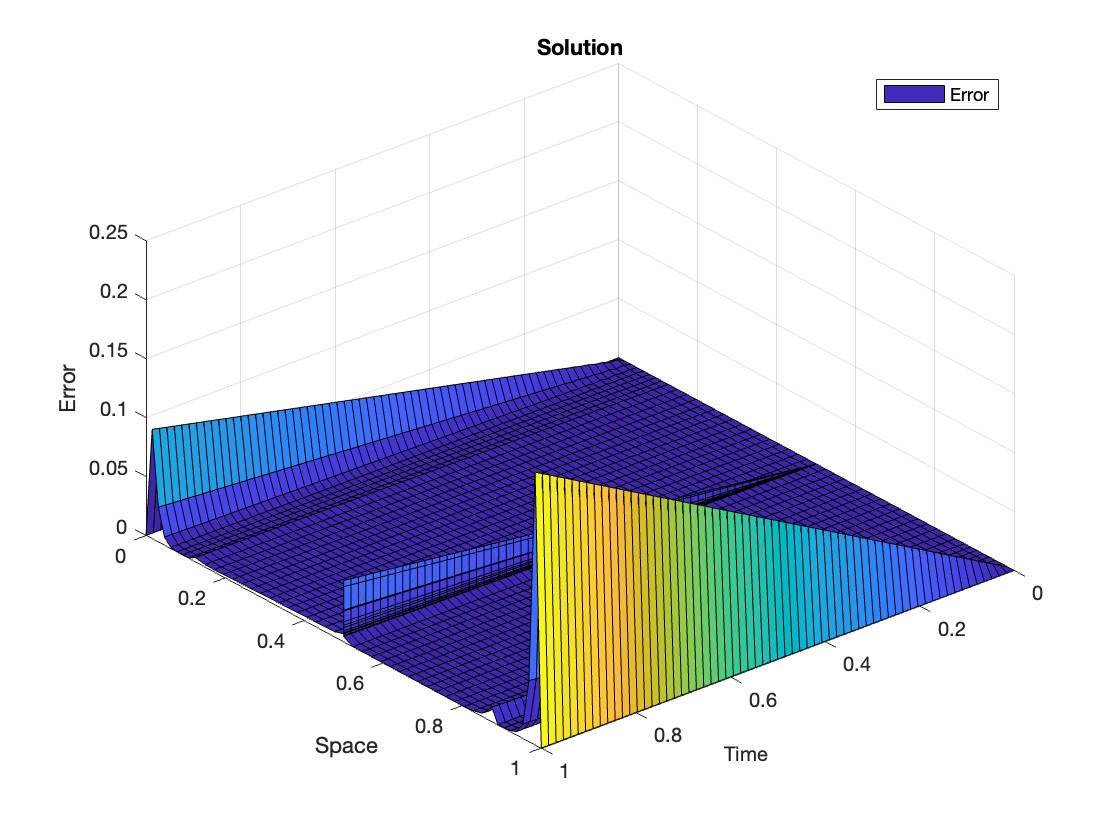}
\caption{Error for $\epsilon=2^{-30}, \mu=2^{-6}$} when $N=64$ for Example \ref{ex-a}.	
\label{fig23}
\end{figure}

\begin{figure}[h]
\centering
\includegraphics[width=0.6\textwidth]{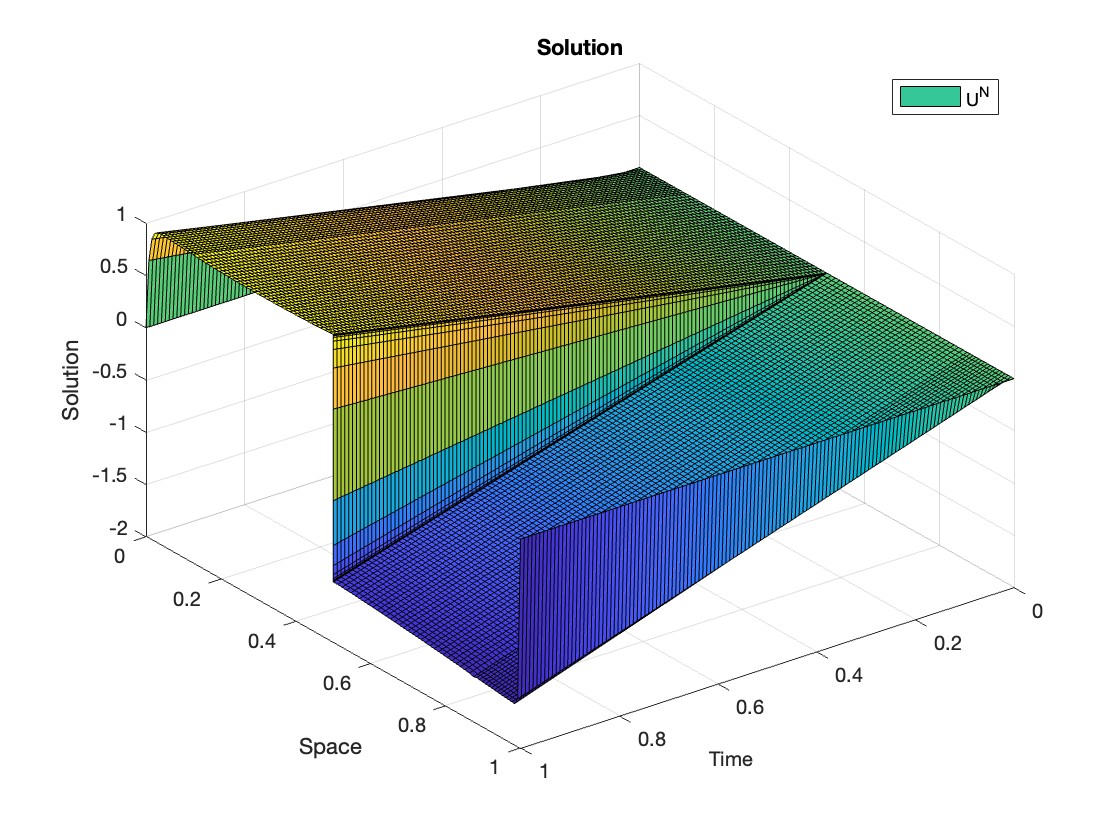}
\caption{Numerical solution for $\epsilon=2^{-22}, \mu=2^{-6}$} when $N=128$ for Example \ref{ex-b}.	
\label{fig31}
\end{figure}
\begin{figure}[h]
\centering
\includegraphics[width=0.6\textwidth]{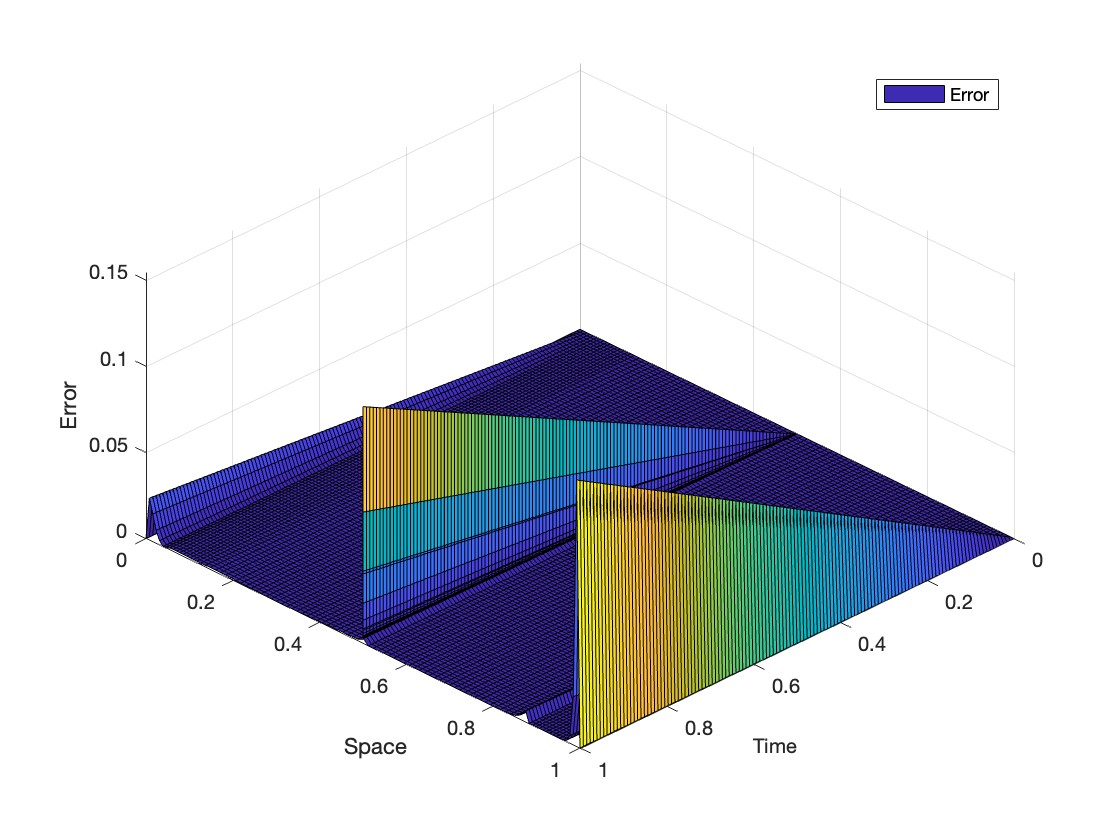}
\caption{Error for $\epsilon=2^{-22}, \mu=2^{-6}$} when $N=128$ for Example \ref{ex-b}.	
\label{fig32}
\end{figure}

\begin{figure}[h]
\centering
\includegraphics[width=0.6\textwidth]{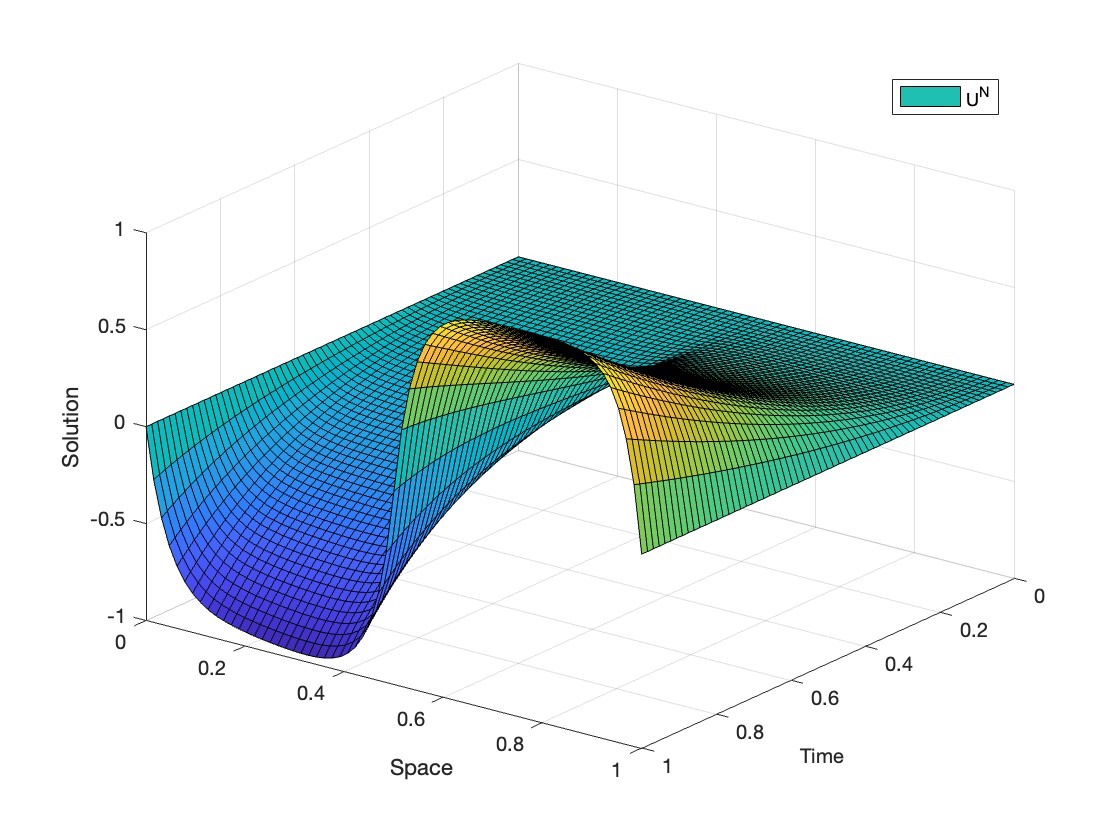}
\caption{Numerical solution for $\epsilon=2^{-8}, \mu=2^{-36}$} when $N=64$ for Example \ref{ex-c}.	
\label{fig33}
\end{figure}
\begin{figure}[h]
\centering
\includegraphics[width=0.6\textwidth]{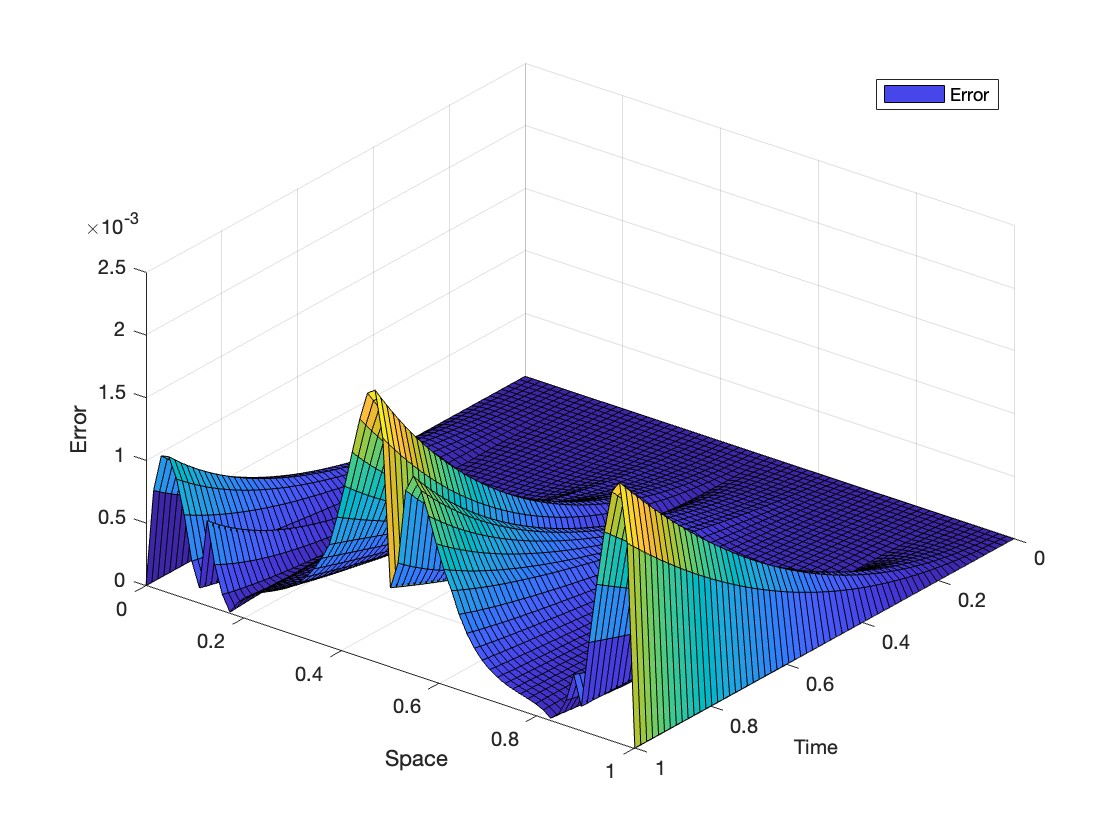}
\caption{Error for $\epsilon=2^{-8}, \mu=2^{-36}$} when $N=64$ for Example \ref{ex-c}.	
\label{fig34}
\end{figure}
The table \eqref{eg11} tabulates the error and numerical order of convergence for the example \eqref{ex-a} for different values of $ \mu $ when  $\epsilon=2^{-6}$. Table \eqref{eg12} presents the numerical results for different values of $\epsilon$ and $\mu=2^{-6}$ for the example \eqref{ex-a}.  
The Figure \eqref{fig11} and \eqref{fig12}  represents the surface plot of numerical solution and error graph of example \eqref{ex-a}  respectively for $\mu=2^{-32}$, $\epsilon=2^{-6}$ and $N=64$. We notice that the maximum error is at the point of discontinuity. Figure \eqref{fig22} and \eqref{fig23} show the numerical solution and error graph for example  \eqref{ex-a} respectively when  $\epsilon=2^{-30}$, $\mu=2^{-6}$ and $N=64$. Here too the maximum error occurs at the point of discontinuity.
In Figure \eqref{fig31} and \eqref{fig32}, the numerical solution and error graph for example  \eqref{ex-b} respectively when  $\epsilon=2^{-22}$, $\mu=2^{-6}$ and $N=128$ is presented. Here the maximum error occurs at the boundaries near $x=1$.
Similarly, The table \eqref{eg13} tabulates the error and numerical order of convergence for the example \eqref{ex-c} for different values of $ \mu $ when  $\epsilon=2^{-8}$.   
The Figure \eqref{fig33} and \eqref{fig34}  represents the surface plot of numerical solution and error graph of example \eqref{ex-c}  respectively for $\epsilon=2^{-8}$, $\mu=2^{-36}$ and $N=64$. We notice that the maximum error is at the point of discontinuity.
\section{Conclusions}
In this article, we studied a singularly perturbed two-parameter parabolic problem characterized by a discontinuous convection coefficient and source term. The solution reveals boundary layers near the domain boundaries and internal layers at the point of discontinuity, induced by the discontinuities in the convection coefficient and source term. For the temporal discretization, we utilized the Crank-Nicolson method on a uniform mesh, achieving second-order accuracy in time. In the spatial direction, we implemented a hybrid difference scheme, which integrates the midpoint method, central difference method, and upwind difference method on a suitably defined Shishkin mesh with a five-point formula at the discontinuity point, ensuring almost second-order accuracy in space. The theoretical analysis was corroborated by numerical results, demonstrating the efficacy and accuracy of the proposed approach.

%\bibliography{reference} 
\end{sloppypar}
\end{document}